\documentclass[12pt,a4paper,english,reqno]{amsart}

\usepackage{amscd, amssymb,}

\theoremstyle{plain}
\newtheorem{thm}{Theorem}[section]
\newtheorem{prp}[thm]{Proposition}
\newtheorem{lem}[thm]{Lemma}
\newtheorem{cor}[thm]{Corollary}
\newtheorem{clm}{Claim}

\newtheorem*{thm-nn}{Theorem}
\newtheorem*{prp-nn}{Proposition}
\newtheorem*{lem-nn}{Lemma}
\newtheorem*{cor-nn}{Corollary}
\newtheorem*{clm-nn}{Claim}
\newtheorem*{cnj-nn}{Conjecture}
\newtheorem*{prb-nn}{Problem}

\theoremstyle{definition}
\newtheorem{dfn}[thm]{Definition}
\newtheorem{exm}[thm]{Example}

\theoremstyle{remark}
\newtheorem{rmk}[thm]{Remark}
\newtheorem{ntn}[thm]{Notation}

\newcommand{\thmref}[1]{Theorem~\ref{#1}}
\newcommand{\prpref}[1]{Proposition~\ref{#1}}

\def\al{\alpha}
\def\be{\beta}
\def\ga{\gamma}
\def\de{\delta}
\def\ep{\varepsilon}
\def\ze{\zeta}
\def\et{\eta}
\def\th{\theta}

\def\la{\lambda}
\def\ro{\rho}
\def\si{\sigma}

\def\ta{\tau}
\def\up{\upsilon}

\def\ph{\phi}
\def\ps{\psi}
\def\Ga{\Gamma}

\def\La{\Lambda}
\def\om{\omega}

\def\Ph{\Phi}
\def\Ps{\Psi}

\def\top{\operatorname{top}} 
\def\soc{\operatorname{soc}}

\def\Im{\operatorname{Im}}

\def\Hom{\operatorname{Hom}}

\def\Aut{\operatorname{Aut}}
\def\End{\operatorname{End}}

\def\mod{\operatorname{mod}}
\def\Mod{\operatorname{Mod}}

\def\obj{\operatorname{obj}}

\def\prj{\operatorname{prj}}

\def\Prj{\operatorname{Prj}}

\def\thick{\operatorname{thick}}
\def\add{\operatorname{add}}

\def\chr{\operatorname{char}}

\def\Rep{\operatorname{Rep}}
\def\Kb{{\mathcal K}^{\text{\rm b}}}

\def\Db{{\mathcal D}^{\text{\rm b}}}

\def\calB{{\mathcal B}}
\def\calC{{\mathcal C}}
\def\calD{{\mathcal D}}

\def\calI{{\mathcal I}}

\def\calK{{\mathcal K}}

\def\calM{{\mathcal M}}

\def\calP{{\mathcal P}}

\def\calS{{\mathcal S}}

\def\bbZ{{\mathbb Z}}
\def\bbP{{\mathbb P}}

\def\k{\Bbbk}
\def\op{^{\text{op}}}

\def\inv{^{-1}}

\def\implies{\text{$\Rightarrow$}\ }

\def\equivalent{\text{$\Leftrightarrow$}\ }

\def\iso{\cong}
\def\ds{\oplus}

\def\udl{\underline}
\def\ovl{\overline}
\def\Ds{\bigoplus}

\def\dsm#1,#2..#3{\bigoplus_{{#1}={#2}}^{#3}}
\def\sm#1,#2..#3{\sum_{{#1}={#2}}^{#3}}

\def\id{1\kern-.25em{\text{{\rm l}}}} 

\def\isoto{\ \raise.8ex\hbox{$^{\sim}$}\kern-.7em\hbox{$\to$}\ } 
\def\Cdot{\centerdot}
\def\down{_{\Cdot}}
\def\up{^{\Cdot}}
\def\ang#1{{\langle #1 \rangle}}
\def\ya#1{\overset{#1}{\longrightarrow}}
\def\blank{\operatorname{-}}

\def\bg{%
\family{cmr}\size{20}{12pt}\selectfont}

\def\bigzerou{%
\smash{\lower1.7ex\hbox{\bg 0}}}

\def\repr[#1;#2;#3;#4;#5]{
\left(
\begin{matrix}#1\\#2\end{matrix}
#3
\begin{matrix}#4\\#5\end{matrix}
\right)}
\def\pmat#1{\begin{pmatrix} #1 \end{pmatrix}}

\def\colim{\varinjlim}

\usepackage[all %
]{xy}
\numberwithin{equation}{section}
\numberwithin{figure}{section}

\def\C{\mathcal{C}}
\def\Aeq{\operatorname{Aeq}}
\def\sic{\operatorname{sic}}
\def\bas{\operatorname{bas}}
\def\fstorbit{/\!_{_{1}}}
\def\sndorbit{/\!_{_{2}}}
\def\freeorbit{/\!_{_{f}}}
\def\oorbit{/\!_{_{o}}}
\def\Fun{\operatorname{Fun}}
\def\Inv{\operatorname{Inv}}
\def\Fgt{\mathrm{Fgt}}

\begin{document}
\title[A generalization of covering functors
and derived equivalences]{A generalization of Gabriel's Galois covering functors
and derived equivalences}
\author{Hideto Asashiba}

\begin{abstract}
Let $G$ be a group acting on a category $\C$.
We give a definition for a functor $F\colon \C\to \C'$
to be a $G$-covering and three constructions of the orbit category $\C/G$,
which generalizes the notion of a Galois covering
of locally finite-dimensional categories
with group $G$ whose action on $\C$ is free and locally bonded defined by Gabriel.
Here $\C/G$ is defined for any category $\C$ and we do not require that
the action of $G$ is free or locally bounded.
We show that a $G$-covering is a universal ``$G$-invariant''
functor and
is essentially given by the canonical functor $\C\to \C/G$.
By using this we improve a covering technique for derived equivalence.
Also we prove theorems describing the relationships between
smash product construction and the orbit category construction
by Cibils and Marcos (2006)
without the assumption that the $G$-action is free.
The orbit category construction by a cyclic group generated by
an auto-equivalence modulo natural isomorphisms
(e.g., the construction of cluster categories) is justified by
a notion of the ``colimit orbit category''.
In addition, we give a presentation of the orbit category of a category
with a monoid action
by a quiver with relations, which enables us to calculate many
examples.
\end{abstract}


\thanks{This work is partially supported by
Grant-in-Aid for Scientific Research (C) 17540036 from JSPS}

\maketitle

\section*{Introduction}

Throughout this paper $G$ is a group (except for sections 8, 9) and $\k$ is a commutative ring,
and all categories, functors and algebras are assumed to be $\k$-linear unless otherwise stated.
(Here a category is called a $\k$-linear category (or  a $\k$-category for short)
if its morphism sets are $\k$-modules
and its compositions are $\k$-bilinear, and we do not assume that it is additive.)
A pair $(\calC, A)$ of a category $\calC$ and a group homomorphism
$A\colon G \to \Aut(\calC)$ is called a category with a $G$-{\em action}
or a $G$-{\em category}, where $\Aut(\C)$ is the group of automorphisms
of $\calC$
(not the group of auto-equivalences of $\C$ modulo natural isomorphisms).
We set $A_{\al}:=A(\al)$ for all $\al\in G$. If there is no confusion
we always (except for sections 8, 9) denote $G$-actions
by the same letter $A$, and simply write
$\calC = (\calC, A)$, and further we usually write
$\al x:= A_{\al}x$, $\al f:= A_{\al}f$ for all $x\in \calC$ and
all morphisms $f$ in $\calC$. 

\subsection*{Classical covering technique}
Let $F\colon \calC \to \calC'$ be a functor with $\calC$ a $G$-category.
The classical setting of covering technique (see e.g., \cite{Ga}) required the following conditions:
\begin{enumerate}
\item
$\C$ is {\em basic} (i.e., $x \ne y\ \implies x\not\iso y$); 
\item
$\C$ is {\em semiperfect}
(i.e., $\C(x,x)$ is a local algebra, $\forall x \in \C$);
\item
$G$-action is {\em free} (i.e.,$1\ne \forall \al \in G, \forall x\in \C, \al x \ne x$); and
\item
$G$-action is {\em locally bounded} (i.e., $\forall x, y \in \C, \{\al\in G\mid \C(\al x, y)\ne 0\}$ is finite).
\end{enumerate}
But these assumptions made it very inconvenient to apply
the covering technique to usual additive categories such as
the bounded homotopy category $\Kb(\prj R)$
of finitely generated projective modules
over a ring $R$ or even the module category $\Mod R$ of $R$
because these categories do not satisfy the condition (2)
and hence we have to construct the full subcategory of
indecomposable objects, which destroys additional structures
like a structure of a triangulated category; and
to satisfy the condition (1) we have to choose a complete set
of representatives of isoclasses of objects that should be
stable under the $G$-action, which is not so easy in practice;
and also the condition (3) is difficult to check in many cases,
e.g., even in the case when we use $G$-actions on the category
$\Kb(\prj R)$ or on $\Mod R$ induced from that on $R$.
These made the proof of the main theorem of a covering technique
for derived equivalences in \cite{Asa97} unnecessarily complicated
and prevented wider applications.
The first purpose of this paper is to generalize the covering technique to remove all these
assumptions.

\subsection*{Orbit categories and covering functors}
Recall that to define a so-called ``root category'' $\Db(\mod H)/[2]$ of a hereditary
algebra $H$ over a field in Happel \cite{Ha} or in Peng-Xiao \cite{PX}
we needed a generalization that removes at least conditions (1) and (2).
It seems, however, even such a simple generalization was not found explicitly in the literature
for a long time.
The definition of root categories given in \cite{PX} works only for itself, and does not give
a general definition of orbit categories.
Nevertheless,
their definition was useful to show that the obtained orbit category
is a triangulated category.
This gave us one of the motivations to start this work.
Recently general definitions of orbit categories was given
in \cite{C-M} by Cibils and Marcos
(let us denote it by $\C \fstorbit G$) and
in \cite{Ke} by Keller (in the case that $G$ is cyclic,
let us denote it by $\C \sndorbit G$).
But we still did not understand the relationship
between the notion of covering functors by Gabriel \cite{Ga}
and the orbit categories defined by them.
We wanted to generalize Gabriel's covering technique as much as possible.
To this end it was necessary to generalize the definition of a covering functor.
In the classical setting the first condition
for a functor $F$ to be
a (Galois) covering functor (with group $G$) is that
$F=FA_{\al}$ for all $\al \in G$.
This leads us naturally to a definition of an {\em invariance adjuster},
a family of natural isomorphisms
$\ph:=(\ph_{\al} \colon F \to FA_{\al})_{\al\in G}$
(see Definition \ref{admissible}).
The pair $(F, \ph)$ is called a ({\em right}) 
$G$-{\em invariant}
functor, further which is called
a $G$-{\em covering} functor if $F$ is a dense functor such that
both
$$
\begin{aligned}
F^{(1)}_{x,y}&\colon \Ds_{\al\in G}\calC(\al x, y) \to \calC'(Fx, Fy),
\quad (f_{\al})_{\al\in G}\mapsto\sum_{\al\in G}F(f_{\al})\cdot\ph_{\al,x}, \text{ and}\\
F^{(2)}_{x,y}&\colon \Ds_{\be\in G}\calC(x, \be y) \to \calC'(Fx, Fy),
\quad (f_{\be})_{\be\in G}\mapsto\sum_{\be\in G}\ph_{\be, y}^{-1}\cdot F(f_{\be})
\end{aligned}
$$
are isomorphisms of $\k$-modules for all $x, y \in \calC$.
In fact, it is enough to require that
either $F^{(1)}_{x,y}$ or $F^{(2)}_{x,y}$ is an isomorphism
for each $x, y \in \calC$.
Roughly speaking the definition of $\calC':= \calC \fstorbit G$ (resp.~$\calC':= \calC \sndorbit G$)
yields by setting all the $F^{(1)}_{x,y}$ (resp.~$F^{(2)}_{x,y}$) to be the identities.
In this paper we give a ``left-right symmetric'' construction
of the orbit category $\C/G$ of $\calC$ by $G$,
which is a direct modification of Gabriel's in \cite{Ga},
and give explicit isomorphisms between $\C/G$, $\C \fstorbit G$ and $\C \sndorbit G$
(Proposition \ref{three-orbit-cats}).
If $F$ has the same property
but it is not necessarily a dense functor,
then $F$ is called a $G$-{\em precovering} functor, which is
useful to induce $G$-covering functors by restricting
the target category $\C'$.
Our characterization (Theorem \ref{chrtrz-cov}) of
$G$-covering functors $F \colon \C \to \C'$
combines the universality among  $G$-invariant functors
and an explicit form of $F$ as the canonical functor
$P\colon \C \to \C/G$ up to equivalences.
We will show that the pushdown (defined as in \cite{Ga}) of
a $G$-covering functor induces $G$-precovering functors
between categories of finitely generated modules
(Theorem \ref{pushdown-precov1})
and between homotopy categories of bounded complexes of finitely generated projective modules (Theorem \ref{pushdown-precov}).
This property will be used to show derived equivalences.

\subsection*{Free action assumption and a categorical generalization of
CM-duality}
Now, in \cite{C-M} Cibils and Marcos gave two definitions of orbit categories.
The first one (let us denote it by $\C \freeorbit G$) is defined only if
the $G$-action is free, and the second one is the orbit category
$\C \fstorbit G$ stated above, called the {\em skew category},
which is defined without the free action assumption.
These two constructions coincide up to categorical equivalences if the $G$-action is free.
But they mainly used $\C \freeorbit G$ and treated only the free action case
in their main discussions in \cite[sections 3, 4]{C-M}, where they
recovered Cohen-Montgomery duality (\cite{Co-Mo}) in the categorical setting (section 3),
and described the module category of $\calC$ by that of $\calC/G$,
which generalizes \cite[Theorem 3.2]{Gr} of Green, and conversely
the module category of $\calC/G$ by that of $\calC$ (section 4).
The second purpose of this paper is to show that all the corresponding statements in
\cite[sections 3, 4]{C-M} hold without the free action assumption.
Namely, (a) we show by elementary proofs that the orbit category construction and
the smash product construction are mutual inverses.
This gives us a full categorical generalization of Cohen-Montgomery duality,
and is regarded as a categorical version of
\cite[Theorems 1.3, 2.2]{Be-dual} of Beattie.
In particular, this gives us a way to make $G$-actions free
up to ``$G$-equivariant equivalences'' (liberalization).
This liberalization can be seen as a special type of the formation
of inflated categories defined by Cibils-Solotar \cite{C-S}.
Further (b) we will show again by elementary proofs that the pullup functor
$P\up\colon \Mod(\calC/G) \to \Mod\calC$ (see section 4 for definition)
induces an {\em isomorphism}
from $\Mod (\calC/G)$ to the category $\Mod^{G} \calC$ of
``$G$-invariant modules''
(see Definition \ref{G-inv}),
and the pushdown functor $P\down\colon \Mod\calC \to \Mod(\calC/G)$
(see section 4 for definition) induces an equivalence from
$\Mod \calC$ to the category $\Mod_{G} (\calC/G)$
of $G$-graded modules and degree-preserving morphisms
(see Definition \ref{graded-mod-hmg}).
The latter gives a generalization of a categorical version of \cite[Theorem 2.6]{Be-gen} of Beattie.
We note that the definition of smash products given in \cite{C-M} is
easy to handle and very useful, and
that we can regard it as a categorical version of the definition of smash products by Quinn \cite{Qu} (when the group is finite), and
it enables us to formulate the covering construction by Green \cite{Gr}, and
recovers the usual smash product of a $\k$-algebra and the $\k$-dual of a group algebra.

In this paper we formulated the categorical version of Cohen-Montgomery duality
as much as possible in the scope of categories.
A more precise investigation needs the notions of
2-categories, 2-functors, etc.,
which will be done in the subsequent paper \cite{Asa09}.

\subsection*{Lax action of a cyclic group}
In \cite{Ke} Keller defined the orbit category $\C \sndorbit G$
only when $G$ is cyclic.
This seems to be mainly because he only needed to construct
an orbit category by a cyclic group generated by an {\em auto-equivalence} $S$ of $\calC$
modulo natural isomorphisms.
As he remarked there, by replacing both $\calC$ and $S$ in a standard way
by a category $\calC'$ and an {\em automorphism} $S'$ of $\calC'$, respectively, we
can form the orbit category $\calC' \sndorbit \ang{S'}$, which he denoted by
$\calC/S$ by abuse of notation and call it the orbit category of
$\calC$ by $S$.
It seems dangerous to forget this remark 
to identify simply as $\calC = \calC'$ and $S = S'$, and
to use the same formula for the definition of $\calC/S$
as if $S$ were an automorphism of $\calC$,
which is not well-defined.
The third purpose of this paper is to give a definition of the orbit category $\calC/S$
directly by replacing neither $\calC$ nor $S$.
More precisely, 
it is known that there are at least two standard ways of replacing the pair $(\calC, S)$.
One way is to replace $\calC$
by a full subcategory consisting of a complete list of
representatives of isoclasses of objects in $\calC$.
Another way is to replace $\calC$ by a category containing more
objects as done in Keller and Vossieck \cite{K-V}.
We realized that the second construction has a form $\calC_{/S}\# \bbZ$ of the
smash product of a $\bbZ$-graded category $\calC_{/S}$ (called
the ``colimit orbit category'' of $\calC$ by $S$)
and the group $\bbZ$.
Applying the generalization of Cohen-Montgomery duality above
we see that the orbit category $\calC/S$
is justified by using the colimit orbit category $\calC_{/S}$. 
When $S$ is an automorphism, of course we have $\calC_{/S}=\calC/\ang{S}$.

The treatment above seems to be the most handy one for the cyclic group case.
For an arbitrary group, we remark that there is another approach to this problem which can be
used even for all categories (not necessarily $\k$-linear ones) instead of groups. 
Namely, when $S$ is an isomorphism, $S$ gives us a functor
$X\colon \bbZ \to \k\text{-}\mathbf{Cat}$ with
$X(*) = \calC$ and $X(1) = S$,
where $\bbZ$ is regarded as a category
with a single object $*$ and with a composition given by the addition,
and $\k\text{-}\mathbf{Cat}$ is the category of all small $\k$-categories.
Even when $S$ is an auto-equivalence, we can define a so-called lax functor
$X\colon \bbZ \to \k\text{-}\mathbf{Cat}$ generalizing the construction above
using the 2-categorical structure of $\k\text{-}\mathbf{Cat}$,
and we can define $\calC/S$ to be the $\k$-linear version of
the so-called Grothendieck construction of $X$ (take direct sums of $\k$-modules
instead of disjoint unions).
This point of view is essential in the forthcoming paper \cite{Asa10}.

\subsection*{Computation by quivers with relations}
Finally, we give a way to compute the first orbit category
$\C \fstorbit G$ using a quiver with relations to apply
theorems in preceding sections.
We generalize it to the monoid case to include a computation of preprojective algebras,
with a hope to have wider applications.

\subsection*{Contents}
The paper is organized as follows.
In section 1, we give a definition of $G$-covering functors
as  $G$-invariant functors with some isomorphism conditions.
In section 2, we construct orbit categories and canonical functors.
Using their universality we prove Theorem \ref{chrtrz-cov},
which will be used to prove the fundamental theorem of
a covering technique for derived equivalences (Theorem \ref{fundamental-thm}) in section 4.
In section 3, we introduce skew group categories in a general setting
as done in the finite group case by Reiten and Riedtmann \cite{RR}.
In section 4, we develop a covering technique for derived equivalences
in our general setting.
In section 5, we prove results in \cite[section 3]{C-M}
without the assumption that the $G$-action is free.
In section 6, we prove the results in \cite{C-M} (Theorems 4.3 and 4.5)
without this free action assumption.
In section 7, we justify the orbit category construction of a category
by a cyclic group generated by an auto-equivalence modulo natural
isomorphisms by introducing a notion of a colimit orbit category.
In section 8, we give a way to compute the first orbit category
$\C \fstorbit G$ using a quiver with relations.
In section 9, we give some examples to illustrate the contents in previous sections,
and include a way to construct a self-injective algebra having a permutation $\si$
as its Nakayama permutation for any given $\si$, which answers a question posed by Oshiro.

In the sequel, the notation $\de_{\al,\be}$ stands for
the Kronecker delta, namely
it has the value 1 if $\al=\be$, and the value 0 otherwise.
By $\calC \simeq \calC'$ (resp.~$\calC \iso \calC'$)
we denote the fact that $\calC$ and $\calC'$
are equivalent (resp.~isomorphic).
\section{Covering functors}

Throughout this section $F\colon \calC \to \calC'$ is a functor
with $\calC$ a $G$-category.
\begin{dfn}
\label{admissible}
An {\em invariance adjuster} of $F$ is a family
$\ph:=(\ph_{\al})_{\al\in G}$ of natural isomorphisms
$\ph_{\al}\colon F\to FA_{\al}$ ($\al\in G$) such that
\begin{enumerate}
\item
$\ph_{1}=\id_{F}$; 
(in fact, this is superfluous, see Remark \ref{superfluous-cond}) and
\item
The following diagram is commutative
for each $\al,\be\in G$:
$$
\xymatrix{
F\ar[r]^{\phi_{\al}}\ar[rd]_{\phi_{\be\al}} & FA_{\al}\ar[d]^{\phi_{\be}A_{\al}}\\ & FA_{\be\al}=FA_{\be}A_{\al},
}
$$
\end{enumerate}
and the pair $(F,\ph)$ is called a
({\em right}) $G$-{\em invariant} functor.

For $G$-invariant functors $(F,\ph)\colon \calC \to \calC'$
and $(F',\ph')\colon \calC \to \calC'$,
a morphism $(F,\ph) \to (F',\ph')$ is a natural transformation
$\et\colon F \to F'$ such that for each $\al\in G$ the following
diagram commutes:
$$
\begin{CD}
F @>{\ph_{\al}}>> FA_{\al}\\
@V{\et}VV @VV{\et A_{\al}}V\\
F' @>>{\ph'_{\al}}> F'A_{\al}.
\end{CD}
$$
\end{dfn}

\begin{rmk}\label{superfluous-cond}
Assume that $\ph:=(\ph_{\al})_{\al\in G}$
in the definition satisfies the condition (2), and let
$x\in \C$ and $\al\in G$.
Then since $\ph_{1,x}:=\ph_{1}x$ is an isomorphism,
the equalities $\ph_{1,x}\ph_{1,x}=\ph_{1,x}$
and $\ph_{\al,x}\ph_{\al\inv,\al x}=\ph_{1,x}$ show the following:
$$\ph_{1,x}=\id_{Fx},
\quad\text{and}\quad \ph_{\al,x}^{-1}=\ph_{\al^{-1},\al x}.$$
Namely, the condition (1) automatically follows from (2).
\end{rmk}

\begin{ntn}
All $G$-invariant functors $\calC \to \calC'$ and all morphisms
between them form a category, which we denote by $\Inv(\calC, \calC')$.
\end{ntn}

\begin{lem}\label{invariant-compn}
Let $F=(F,\ph)$
be a $G$-invariant functor, and $H\colon \calC' \to \calC''$ a functor.
Then $(HF, H\ph)\colon \calC \to \calC''$ is a $G$-invariant functor,
where $H\ph:= (H\ph_{\al})_{\al\in G}$.
\end{lem}

\begin{proof}
Straightforward.
\end{proof}

\begin{ntn}
Let $F=(F,\ph)$
be a  $G$-invariant functor, and let $x,y\in \C$.
Then we define homomorphisms
$F_{x,y}^{(1)}$ and $F_{x,y}^{(2)}$
of $\k$-modules as follows:
\begin{align*}
F_{x,y}^{(1)}&\colon\bigoplus_{\al\in G}\C(\al x,y)\to \C'(Fx,Fy),\ (f_{\al})_{\al\in G}\mapsto\sum_{\al\in G}F(f_{\al})\cdot\ph_{\al,x};\\
F_{x,y}^{(2)}&\colon\bigoplus_{\be\in G}\C(x,\be y)\to \C'(Fx,Fy),\ (f_{\be})_{\be\in G}\mapsto\sum_{\be\in G}\ph_{\be^{-1},\be y}\cdot F(f_{\be}).
\end{align*}
\end{ntn}

\begin{prp}
Let $F=(F,\ph)$
be a  $G$-invariant functor, and let $x,y\in \C$.
Then $F_{x,y}^{(1)}$ is an isomorphism if and only if
$F_{x,y}^{(2)}$ is.
\end{prp}

\begin{proof}
This follows from the following commutative diagram
$$
\xymatrix{
\Ds_{\al\in G}\C(\al x, y) \ar[r]^{F_{x,y}^{(1)}}
\ar[d]_{t}^{\wr}& \C'(Fx,Fy) \ar@{=}[dd]\\
\Ds_{\al\in G}\C(\al\inv x, y) \ar[d]_{(\al)_{\al\in G}}^{\wr}\\
\Ds_{\al\in G}\C(x, \al y) \ar[r]_{F_{x,y}^{(2)}} &\C'(Fx, Fy),
}
$$
where $t$ is defined by
$t((f_{\al})_{\al\in G}):= (f_{\al\inv})_{\al\in G}$,
which is clearly an isomorphism of $\k$-modules.
\end{proof}

\begin{dfn}
Let $F=(F, \ph)$ be a  $G$-invariant functor. Then

(1) $F=(F,\ph)$ is called a $G$-\emph{precovering} if for any
$x,y\in \C$ the $\k$-homomorphism $F_{x,y}^{(1)}$ is an isomorphism
(equivalently, if $F_{x,y}^{(2)}$ is an isomorphism).

(2) $F=(F,\ph)$ is called a $G$-\emph{covering} if $F$ is a
$G$-precovering and $F$ is \emph{dense}, in the sense that for any
$x'\in \C'$ there exists an $x\in \C$ such that $x'$ is isomorphic
to $Fx$ in $\C'$.
\end{dfn}

\section{Orbit categories}

\begin{dfn}
The \emph{orbit category} $\C/G$ of $\C$ by $G$ is defined as follows.

(1) The class of objects of $\C/G$ is equal to that of $\C$.

(2) For each $x,y\in \C/G$ we set\[
(\C/G)(x,y):=\left(\Pi'(x,y)\right)^{G},\]
where \[
\Pi'(x,y):=\{ f=(f_{\be,\al})_{(\al,\be)}\in\prod_{(\al,\be)\in G\times G}\C(\al x,\be y)\mid f\textrm{ is row finite and column finite}\},\]
and $(\textrm{-})^{G}$ stands for the set of $G$-invariant elements,
namely\[
\left(\Pi'(x,y)\right)^{G}:=\{(f_{\be,\al})_{(\al,\be)}\in\Pi'(x,y)\mid\forall\ga\in G,f_{\ga\be,\ga\al}=\ga(f_{\be,\al})\}.\]
In the above, $f$ is said to be \emph{row finite} (resp. \emph{column
finite}) if for any $\al\in G$ the set $\{\be\in G\mid f_{\al,\be}\ne 0\}$
(resp. $\{\be\in G\mid f_{\be,\al}\ne0\}$) is finite.

(3) For any composable morphisms $x\ya{f}y\ya{g}z$ in $\C/G$ we set\[
gf:=\left(\sum_{\ga\in G}g_{\be,\ga}\cdot f_{\ga,\al}\right)_{(\al,\be)\in G\times G}\in(\C/G)(x,z).\]
\end{dfn}

\begin{rmk}\label{oldorbit}
(1) In the usual definition of the orbit category,
one sets $\obj(\calC/G):= \{Gx\mid x\in \obj(\calC)\}$,
where $Gx:=\{\al x\mid \al\in G\}$ for all $x\in \calC$.
But this makes a trouble when $G$-action is not free.
This was changed as in (1) above, which enabled us to remove the
classical assumption that the $G$-action is free.
Nevertheless note that if the $G$-action is free, we can define another orbit category
$\calC\oorbit G$ by setting
$$
\obj(\calC\oorbit G):= \{Gx\mid x\in \obj(\calC)\}
$$
as usual, and for each $x, y \in \calC$
$$
(\calC\oorbit G)(Gx, Gy):=\{ f=(f_{b,a})\in\prod_{(a,b)\in Gx\times Gy}\C(a,b)\mid f\textrm{ is row finite and column finite}\}^{G}
$$
with the similar composition as above.

(2) As in (2) above, by considering only row finite and column finite
matrices we could remove the classical assumption
that the $G$-action is locally bounded.
But if we further require the condition that the Hom-spaces
$(\calC/G)(x,y)$ are finitely generated $\k$-modules,
we need this locally bounded action assumption again.
\end{rmk}

\begin{prp}
$\C/G$ is a $\k$-category.
\end{prp}
\begin{proof}
For each $x\in \calC$ the identity $\id_{x}$ in $\C/G$ is given by
\begin{equation}\label{identity}
\id_{x}=(\delta_{\al,\be}\id_{\al x})_{\al,\be\in G}.
\end{equation}
The rest is easy to verify and is left to the reader.
\end{proof}

\begin{dfn}
The {\em canonical} functor $P\colon \C \to \C/G$ is defined by
$P(x):=x$, and $P(f):=(\de_{\al,\be}\,\al f)_{(\al,\be)}$ for all $x, y \in \C$ and all $f\in \C(x,y)$.
\end{dfn}

\begin{dfn}
For each $\mu\in G$ and each $x \in \C$ define
$\ph_{\mu, x}:= (\de_{\al, \be\mu}\id_{\al x})_{(\al,\be)}\in (\C/G)(P x, P \mu x)$, and set
$\ph_{\mu}:=(\ph_{\mu,x})_{x\in \C}\colon P \to P A_{\mu}$.
Then $\ph:= (\ph_{\mu})_{\mu\in G}$ is an invariance adjuster of $P$, and
hence $P=(P,\ph)$ is a  $G$-invariant functor.
\end{dfn}

\begin{prp}\label{pi-universal}
$P=(P, \ph) \colon \C \to \C/G$ has the following properties.
\begin{enumerate}
\item
$P=(P, \ph)$ is a $G$-covering functor;
\item
$P=(P, \ph)$ is {\em universal} among  $G$-invariant functors
from $\C$, namely,
for each  $G$-invariant functor $E=(E,\ps) \colon \calC \to \calC'$,
there exist a unique $($up to isomorphism$)$ functor
$H\colon \C/G \to \C'$ such that
$(E,\ps) \iso (HP,H\ph)$ as  $G$-invariant functors; and
\item
$P=(P, \ph)$ is {\em strictly universal} among  $G$-invariant functors
from $\C$, namely,
for each  $G$-invariant functor $E=(E,\ps) \colon \calC \to \calC'$,
there exist a $($really$)$ unique functor
$H\colon \C/G \to \C'$ such that $(E,\ps) = (HP,H\ph)$.
\end{enumerate}
\end{prp}

\begin{proof}
(1) By definition $P$ is dense.
Let $x, y \in \calC$. We have only to show that
$$
P^{(1)}_{x,y}\colon \Ds_{\al\in G}\calC(\al x,y)
\to (\calC/G)(x,y)
$$
is an isomorphism of $\k$-modules.
By definitions of $P$ and $\ph$ a direct calculation shows that
\begin{equation}
\label{form-pi1}
P^{(1)}_{x,y}((f_{\al})_{\al})= (\mu(f_{\mu\inv\la}))_{(\la,\mu)}
\end{equation}
for all $f=(f_{\al})_{\al}\in \Ds_{\al\in G}\calC(\al x,y)$.
Now define a $\k$-homomorphism
$$
S^{(1)}_{x,y}\colon (\calC/G)(x,y) \to \Ds_{\al\in G}\calC(\al x,y)
$$
by $S^{(1)}_{x,y}((f_{\be,\al})_{(\al,\be)}):=(f_{1,\al})_{\al}$,
which is easily seen to be the inverse of $P^{(1)}_{x,y}$
by using the equality \eqref{form-pi1},
and hence $P^{(1)}_{x,y}$ is an isomorphism.

(2) and (3)
Let $E=(E,\ps)\colon \calC \to \calC'$ be a  $G$-invariant functor.
Define a functor $H\colon \calC/G \to \calC'$ as follows.
For each $x, y \in \calC/G$ and
each $f=(f_{\be,\al})_{(\al,\be)}\in (\calC/G)(x,y)$,
let $H(x):=E(x)$ and
$H(f):=(E^{(1)}_{x,y}S^{(1)}_{x,y})(f)
=\sum_{\al\in G} E(f_{1,\al})\ps_{\al,x}$.
Then we have a commutative diagram
$$
\xymatrix{
\Ds_{\al\in G}\calC(\al x,y) \ar[rr]^{E^{(1)}_{x,y}}
\ar[rd]_{P^{(1)}_{x,y}}&&\calC'(Ex, Ey)\\
&(\calC/G)(x,y).\ar[ru]_H
}
$$
We show that $H$ is a functor.
First for each $x\in \calC/G$,
using \eqref{identity} and the definition of $H$,
a direct calculation shows that
$H(\id_x)=E(\id_x)$.
Next, let $x\ya{f} y \ya{g} z$ be composable morphisms in $\calC/G$.
Then using the naturality of $\ps_{\be}$ ($\be\in G$)
and the fact that $\ps$ is an invariance adjuster,
we have
$H(g)H(f)=\sum_{\al,\be\in G}E(g_{1,\be})E(f_{\be,\be\al})\ps_{\be\al,x}$,
the right hand side of which is easily seen to be equal to $H(gf)$.
Therefore $H(gf)=H(g)H(f)$.
Further, the $\k$-linearity of $H$ is clear from definition, and hence
$H$ is a functor.

Next let $\si\colon \calC(x,y) \to \Ds_{\al\in G}\calC(\al x, y)$
be the inclusion (more precisely, it is defined by
$\si(f):=(\de_{1,\al}f)_{\al}$ for all $f\in \calC(x,y)$).
Then as easily seen $P = P^{(1)}_{x,y}\si$ and
$E =E^{(1)}_{x,y}\si$.  Thus the commutative diagram above shows that
$E=HP$ (the equality on objects is clear from definitions).
Further the definitions of $H$ and $\ph$ also show that $H\ph=\ps$.
Hence $(E,\ps) = (HP,H\ph)$.
This shows the existence of $H$ in both (2) and (3).

Finally, we show the uniqueness of $H$ in the sense of (2).
Assume that there is a functor $H'\colon \calC/G \to \calC'$
such that $(E,\ps) \iso (H'P,H'\ph)$.
Then there is a natural isomorphism $\et\colon E \to H'P$ such that
for each $\al \in G$ the following diagram commutes: 
\begin{equation}\label{compati-et}
\begin{CD}
E @>{\ps_{\al}}>> E A_{\al}\\
@V{\et}VV @VV{\et A_{\al}}V\\
H'P @>>{H'\ph_{\al}}> H'PA_{\al}
\end{CD}
\end{equation}
We have to show that there is a natural isomorphism between
$H$ and $H'$.
Now for each $x\in \calC$ we have an isomorphism
$\et_x\colon Hx=Ex \to H'P x = H'x$.
Using this define a family $\ze$ of isomorphisms by $\ze:=(\et_x)_x$.
Then this gives a desired natural isomorphism $\ze\colon H \to H'$.
Indeed, let $f:= (f_{\be,\al})_{(\al,\be)}\colon x \to y$ be in $\calC/G$.
It is enough to show the commutativity of the following diagram:
\begin{equation}\label{necessary-comm}
\begin{CD}
Hx @>{\et_x}>> H'x\\
 @V{H(f)}VV  @VV{H'(f)}V\\
Hy @>>{\et_y}> H'y
\end{CD}
\end{equation}
First, for each $\al \in G$ the naturality of $\et$ gives us the following:
$$
\et_yE(f_{1,\al})=H'P(f_{1,\al})\et_{\al x}
$$
Next, \eqref{compati-et} shows the following.
$$
\et_{\al x}\ps_{\al,x}=H'(\ph_{\al,x})\et_x
$$
Using these equalities in this order we have
$$
\begin{aligned}
\et_y H(f)&= \sum_{\al\in G}\et_y E(f_{1,x})\ps_{\al,x}\\
&= \sum_{\al\in G}H'P(f_{1,\al})\et_{\al x}\ps_{\al,x}\\
&= \sum_{\al\in G}H'P(f_{1,\al})H'(\ph_{\al,x})\et_x\\
&= H'(\sum_{\al\in G}P(f_{1,\al})\ph_{\al,x})\et_x\\
&=H'(P^{(1)}_{x,y}S^{(1)}_{x,y}(f))\et_x \\
&=H'(f)\et_x,
\end{aligned}
$$
which shows the commutativity of \eqref{necessary-comm}.

The uniqueness of $H$ in the sense of (3) follows
from the argument above as a special case that
$\et_{x} = \id_{Hx}$ for all $x \in \calC/G$. 
\end{proof}

The following was pointed out by B.~Keller as a comment
about the proposition above,
which will be used in section 6 (Theorem \ref{pullup-iso}).
\begin{cor}\label{2-universality}
The canonical functor $(P,\ph) \colon \calC \to \calC/G$ is
2-universal among $G$-invariant functors from $\calC$, i.e.,
the induced functor
$$
(P,\ph)^*\colon \Fun(\calC/G, \calC') \to \Inv(\calC, \calC')
$$
is an isomorphism of categories for all categories $\calC'$, where
$\Fun(\calC/G, \calC')$ is the category of functors from
$\calC/G$ to $\calC'$.
\end{cor}

\begin{proof}
By Proposition \ref{pi-universal}(3), $(P,\ph)^*$ is bijective
on objects. Since $P\colon \calC \to \calC/G$ is dense,
$(P,\ph)^*$ is fully faithful by a general theory.
\end{proof}

\begin{exm}
\label{adjuster-representation}
Let $\calC = \k$ be a field and let $\calC' = \k\text{-}\Mod $ the category of (left) $\k$-vector spaces.
Assume that the $G$-action on $\calC$ is trivial.
Then the orbit category $\calC/G$ turns out to be the usual group algebra $\k G$,
$\Fun(\calC/G, \calC') = \k G\text{-}\Mod $ is the category of left $\k G$-modules, and
$\Inv(\calC, \calC') = \Rep_{\k}G$ is the category of $\k$-representations of $G$.
In this case the isomorphism above coincides with the well-known isomorphism
$\k G\text{-}\Mod  \iso \Rep_{\k}G$.
\end{exm}

$G$-covering functors are characterized as follows
(cf. the definition of Galois covering in \cite{Ga}.)

\begin{thm}
\label{chrtrz-cov}
Let $F=(F,\ps)$ be a  $G$-invariant functor.
Then the following are equivalent.
\begin{enumerate}
\item
$F=(F, \ps)$ is a $G$-covering;
\item
$F=(F,\ps)$ is a $G$-precovering that is universal among $G$-precoverings
from $\C$;
\item
$F=(F,\ps)$ is universal among  $G$-invariant functors from $\C$;
\item
There exist an equivalence $H\colon \C/G \to \C'$
such that $(F,\ps) \iso (HP, H\ph)$ as  $G$-invariant functors; and
\item
There exist an equivalence $H\colon \C/G \to \C'$
such that $(F,\ps) = (HP, H\ph)$.
\end{enumerate}
\end{thm}

\begin{proof}
(1) \equivalent (4).
If the statement (1) holds,
then the following holds by \prpref{pi-universal}(2):

($*$) There exist a functor $H\colon \C/G \to \C'$
and an isomorphism $\et\colon (F,\ps) \to (HP, H\ph)$ of
 $G$-invariant functors.

This also follows from the statement (4) trivially.
Hence to show the equivalence of (1) and (4), it is enough
to show that $F$ is a $G$-covering if and only if $H$ is an
equivalence in the setting of ($*$).
More precisely we show that (a) $F$ is dense if and only if so is $H$;
and (b) $F$ is a $G$-precovering if and only if $H$ is fully faithful.
Let $x \in \C'$.
For each $y\in \obj(\C)=\obj(\C/G)$ we have an isomorphism
$\et_y\colon Fy \to HP y=Hy$ in $\C'$.
Hence $x \iso Fy$ if and only if $x \iso Hy$.
This shows the statement (a).
Now let $x, y \in \C$ and $(f_{\al})_{\al} \in \Ds_{\al\in G}\C(\al x,y)$.
Then we have a commutative diagram
$$
\begin{CD}
Fx @>{\ps_{\al,x}}>> F\al x @>{F(f_{\al})}>> Fy\\
@V{\et_x}VV  @VV{\et_{\al x}}V @VV{\et_y}V\\
HP x @>>{H\ph_{\al,x}}> HP\al x @>>{HP(f_{\al})}> HP y,
\end{CD}
$$
which yields the following commutative diagram:
$$
\begin{CD}
\Ds_{\al\in G}\C(\al x,y) @>{P^{(1)}_{x,y}}>> \C/G(x,y)\\
@V{F^{(1)}_{x,y}}VV @VV{H_{x,y}}V\\
\C'(Fx,Fy) @>>{\et_y(\blank)\et_x\inv}> \C'(Hx, Hy),
\end{CD}
$$
where $H_{x,y}$ is the restriction of $H$ to $\C/G(x,y)$.
Since the horizontal maps are isomorphisms,
the commutativity of this diagram shows
that $F^{(1)}_{x,y}$ is an isomorphism
if and only if $H_{x,y}$ is.
Hence (b) holds.

(2) \equivalent (4)
Note that $P=(P,\ph)$ is also a $G$-precovering.
Since all $G$-precoverings from $\C$ are
 $G$-invariant functors from $\C$,
$P$ has the universal property also among
$G$-precoverings from $\C$, by which 
this equivalence is obvious.

(3) \equivalent (4).
Since $P= (P,\ph)$ is also universal
among  $G$-invariant functors from
$\C$, this equivalence is obvious.

(5) \equivalent (4).
The implication ``(5) \implies (4)'' is trivial.
If (4) holds, then (1) holds and by \prpref{pi-universal}(3)
we have a functor $H\colon \C/G \to \C'$
such that $(F,\ps) = (HP, H\ph)$.
This $H$ is an equivalence by the argument above.
\end{proof}

The author learned the following construction from Keller \cite{Ke}.

\begin{dfn}[Cibils-Marcos, Keller]
(1) An orbit category $\C\fstorbit G$
is defined as follows.
\begin{itemize}
\item
$\obj(\C\fstorbit G):=\obj(\C)$;
\item
$\forall x, y \in G$, $\C\fstorbit G(x,y):=\Ds_{\al\in G}\C(\al x, y)$; and
\item
For $x\ya{f} y \ya{g} z$ in $\C\fstorbit G$,
$gf:=(\sum_{\al,\be\in G; \be\al=\mu}g_{\be}\cdot \be(f_{\al}))_{\mu\in G}$.
\end{itemize}
(2) Similarly another orbit category $\C\sndorbit G$ is defined as follows.
\begin{itemize}
\item
$\obj(\C\sndorbit G):=\obj(\C)$;
\item
$\forall x, y \in G$, $(\C\sndorbit G)(x,y):=\Ds_{\be\in G}\C(x,\be y)$; and
\item
For $x\ya{f} y \ya{g} z$ in $\C\sndorbit G$,
$gf:=(\sum_{\al,\be\in G; \al\be=\mu}\al(g_{\be})\cdot f_{\al})_{\mu\in G}$.
\end{itemize}
Note that $\C\sndorbit G=(\C\op\fstorbit G)\op$.
\end{dfn}

\begin{prp}\label{three-orbit-cats}
We have isomorphisms of categories
$$\C\fstorbit G \iso \C/G \iso  \C\sndorbit G.$$
\end{prp}

\begin{proof}
The isomorphisms $S^{(1)} \colon \calC/G \to \calC\fstorbit G$ and
$S^{(2)} \colon \calC/G \to \calC\sndorbit G$ are given by identities on objects,
and on morphisms by
$$\xymatrix{
\calC\fstorbit G(x,y)&\ar[l]_{S^{(1)}_{x,y}}\calC/G(x,y)
\ar[r]^{S^{(2)}_{x,y}}&\calC\sndorbit G(x,y)
}
$$
for all $x,y\in \calC$,
where $S^{(1)}_{x,y},S^{(2)}_{x,y}$ are defined by
$$\xymatrix{
(f_{1,\al})_{\al\in G}&\ar@{|->}[l](f_{\be,\al})_{(\al,\be)\in G\times G} \ar@{|->}[r]& (f_{\be,1})_{\be\in G} 
}$$
for all $(f_{\be,\al})_{(\al,\be)\in G\times G} \in \calC/G(x,y)$.
It is easy to verify that $S^{(1)}$ and $S^{(2)}$ are functors.
As in the proof of \prpref{pi-universal}(1), $S^{(i)}_{x,y}$
has the inverse $P^{(i)}_{x,y}$ for $i=1\text{ and }2$, and hence
$S^{(1)}$ and $S^{(2)}$ are isomorphisms of categories.
\end{proof}

\begin{exm}
\label{sk-gp-alg}
Let $R$ be an algebra, and $G \le \Aut(R)$. Regard $R$ as a category with only one object.
Then $R/G \iso R\fstorbit G \iso R*G$ (skew group algebra).
Indeed, an isomorphism $R\fstorbit G \to R*G$ is given by
$(f_{\al})_{\al} \mapsto \sum_{\al} f_{\al}*\al$; and
the multiplication rule $g_{\be}\cdot f_{\al}=g_{\be}\cdot\be(f_{\al})$ in $R\fstorbit G$
corresponds to the rule
$(g_{\be}*\be)(f_{\al}*\al)=g_{\be}\cdot \be(f_{\al})*\be\al$
in $R*G$ for all $\al, \be \in G$ and $f_{\al}, g_{\be} \in R$.
\end{exm}

\begin{rmk}
Even when $G$ is a monoid,
the two orbit categories $\calC\fstorbit G$ and
$\calC \sndorbit G$ are defined although the orbit category
$\calC/G$ is not well-defined in general.
But in that case these are not isomorphic to each other
in general.
For instance, let $G$ be the monoid $\ang{\al\mid \al^2=\al}$ and
$\calC:= \k[x]/(x^2)$ with a $G$-action defined by
$\al(a+b\bar{x}):=a$ for all $a,b\in \k$, where $\k$ is a field and $\bar{x}:=x+(x^2)$.
Then $\calC/G$ is not well-defined but $\calC\fstorbit G$, $\calC\sndorbit G$
are defined and have the forms
$\calC\fstorbit G \iso \k\ang{x,y}/(x^2, y^2-y,yx)$ and
$\calC\sndorbit G \iso \k\ang{x,y}/(x^2, y^2-y,xy)$.
A direct calculation shows that $\calC\fstorbit G \not\iso \calC\sndorbit G$.
\end{rmk}

\begin{rmk}\label{rmk:orbit-1}
Cibils and Marcos \cite{C-M} call $\C\fstorbit G$
the {\em skew category} and denote it by $\C[G]$, and
they have the same opinion that this (or its {\em basic} category,
see Definition \ref{basic}) can be considered as a substitute for
the orbit category in the case that $G$-action on $\C$ is not free.
(Cf. Remark \ref{rmk:skw}.)
\end{rmk}

\section{Skew group categories}

The following construction is well-known (see Freyd \cite[p.\ 61, Exercise B]{Fr}, \cite[2.1, Example 7]{Ga-Ro} for instance).

\begin{dfn}
The {\em split idempotent completion} of a category $\calC$ is
the category $\sic(\calC)$ defined as follows.
Objects of  $\sic(\calC)$ are the pairs $(x, e)$ with $x \in \calC$ and $e^{2}=e\in \calC(x,x)$.
For two objects $(x, e), (x', e')$ of $\sic(\calC)$,
the set
 of morphisms from $(x,e)$ to $(x',e')$ is given by
$\sic(\calC)((x,e),(x',e')):=\{f \in \calC(x,x') \mid f=e'fe\}$, and the composition is given by that of $\calC$.
\end{dfn}

\begin{rmk}
It is obvious that all idempotents in $\sic(\calC)$ split, and that
the canonical embedding $\si_{\calC}\colon \calC \to \sic(\calC)$
sending each morphism $f\colon x \to y$ in $\calC$
to $f\colon (x,\id_{x}) \to (y, \id_{y})$ is universal among functors from $\calC$ to a category
with all idempotents split.
\end{rmk}

\begin{dfn}
Contravariant functors from $\calC$ to the category $\Mod \k$
of $\k$-modules are called (right) $\calC$-{\em modules}.
The class of them together with the natural transformations
between them forms a category, which is denoted by $\Mod \calC$.
\end{dfn}

\begin{prp}\label{sic-Morita}
The canonical embedding $\si_{\calC}\colon\calC \to \sic(\calC)$ induces an equivalence of
module categories
$\si\colon \Mod \sic(\calC) \to \Mod \calC$.
Thus $\calC$ and $\sic(\calC)$ are Morita equivalent. 
\end{prp}

\begin{proof}
A quasi-inverse  $\ta\colon \Mod  \calC  \to \Mod \sic(\calC)$ of $\si$ is given as follows.
Let $\la\colon M \to M'$ be in $\Mod\calC$.
For each $(x,e) \in \sic(\calC)$ with $x\in \calC$ and $e=e^2\in \calC(x,x)$,
$(\ta M)(x,e):= \Im M(e) \ (\le M(x))$; and 
$(\ta\la)_{(x,e)}:= \la_{x}|_{\Im M(e)}$, the restriction of $\la_x$.
It is easy to see that these are well-defined and that $\ta$ is a quasi-inverse of $\si$.
\end{proof}

\begin{dfn}\label{basic}
A full subcategory $\calC'$ of a category $\calC$ is called a {\em basic} category of $\calC$
if the objects of $\calC'$
form a complete set of representatives of isoclasses of objects of $\calC$.
In this case it is obvious that the canonical embedding
$\calC' \to \calC$ is an equivalence,
and hence basic categories of $\calC$ are pairwise isomorphic.
We take one of them and denote it by $\bas(\calC)$.
We also choose a quasi-inverse of the canonical embedding
 $\iota_{\bas(\C)}\colon \bas(\calC) \to \calC$ and denote it by 
 $\ro_{\calC}\colon \calC \to \bas(\calC)$.
\end{dfn}

\begin{dfn}\label{dfn:skew-gp-cat}
Assume that a group $G$ acts on a category $\calC$.
Then the category
$\calC * G :=\bas(\sic(\calC/G))$ is called a {\em skew group category} of $\calC$ by $G$.
We denote the composite of the functors
$\calC \ya{P}  \calC/G \xrightarrow{\si_{\calC/G}} \sic(\calC/G) \xrightarrow{\ro_{\sic(\calC/G)}} \calC * G$ also by $P$.
Note that $\calC/G$ and $\calC * G$ are Morita equivalent by \prpref{sic-Morita}.
\end{dfn}

\begin{rmk}\label{rmk:skw}
The name ``skew group category'' came from the fact described in Example \ref{sk-gp-alg}.
When $G$ is a finite group the definition above coincides with that given in Reiten-Riedtmann \cite{RR}. (Cf. Remark \ref{rmk:orbit-1}.)
\end{rmk}

\begin{rmk}
\label{rmk:autoequivalence}
We make the following remark on auto-equivalences.
Consider the case that the $G$-action on $\C$
is given by
auto-equivalences of $\C$ modulo natural isomorphisms:
$$G \to \Aeq(\C)/\!\cong.$$
An important example is given by the construction of cluster categories,
where $G$ is cyclic.
When $G$ is cyclic, say $G=\ang{\bar{F}}$ with $\bar{F} \in \Aeq(\C)/\!\!\cong$
and $F \in \bar{F}$,
the orbit category $\C/F:=\C/\ang{\bar{F}}$ of $\C$ by $\ang{\bar{F}}$ can be defined by
setting $\C/F:= \bas(\C)/\ang{F'}$,
where $F':= \ro_{\C} \circ F \circ \iota_{\bas(\C)}$ is an isomorphism of $\bas(\C)$
(see Definition \ref{basic} for notations).

But if $G$ is not cyclic,
then this standard construction {\em does not work} in general.
An alternative construction will be given later
(see section \ref{section:colim-orb-cat}).
\end{rmk}

Here we give a definition of skew monoid categories (or algebras)
by generalizing the notion of skew group categories.
Recall that a category $\calC$ defines the corresponding algebra
$\ds \calC$ by
\begin{equation}\label{cat-alg}
\ds \calC:= \Ds_{x,y\in \calC}\calC(x,y),
\end{equation}
where elements $f$ of the right hand side is regarded as matrices
$f=(f_{y,x})_{x,y\in \calC}$ and the multiplication is given by the usual matrix multiplication (see e.g.\ \cite{Ga-Ro}).

\begin{dfn}
\label{dfn:skew-monoid-cat}
Let  $\calC$ be a category and $G$ a monoid acting on $\calC$.
Here we assume that the $G$-action on $\calC$ is given by a
homomorphism $G \to \End(\calC)$, where
$\End(\calC):=\{f\colon\calC \to \calC \mid f \text{ is a functor}\}$.
In the case that $G$ contains 0, we add the zero object into $\calC$ and we allow that
$f(x)=0$ for some $f\in \End(\calC)$ and $x \in \calC$.
We define a {\em skew monoid category} $\calC * G$ by
$$\calC * G:= \bas(\sic(\calC\fstorbit G)).$$
A {\em skew monoid algebra} $(\ds \calC) * G$ is defined by
$$(\ds \calC) * G:= \ds(\calC * G) = \ds\bas(\sic(\calC\fstorbit G)).$$
\end{dfn}

\section{Pushdown functors and derived equivalences}

\begin{dfn}
Let $R$ be a category.
\begin{enumerate}
\item
The full subcategory of $\Mod R$ consisting of projective objects
is denoted by $\Prj R$.
Note that an $R$-module $X$ is projective
if and only if $X$ is isomorphic to a direct summand of a direct sum
of representable functors $R(\blank, x)$ ($x \in R$).
\item
An $R$-module $X\in \Mod R$ is called {\em finitely generated}
if there exists an epimorphism from a finite direct sum of representable
functors to $X$.
Note that $X$ is a finitely generated projective $R$-module if and only if
$X$ is isomorphic to a direct summand of a finite direct sum of representable functors.
The full subcategory of $\Prj R$ consisting of
finitely generated projective $R$-modules is denoted by $\prj R$.
The full subcategory of $\Mod R$ consisting of finitely generated
$R$-modules is denoted by $\mod R$.
\item
The homotopy category of $\Prj R$ is denoted by
$\calK(\Prj R)$ and the full subcategory of
$\calK(\Prj R)$ consisting of bounded complexes of
finitely generated projectives is denoted by
$\Kb(\prj R)$.
\end{enumerate}
\end{dfn}
\begin{dfn}
Let $G$ be a group acting on a category $R$, and
$P\colon R \to R/G$ the canonical functor.
\begin{enumerate}
\item
The functor $P^{\Cdot}: \Mod R/G \to \Mod R$
defined by $P^{\Cdot}M:= M\circ P$
for all $M \in \Mod R/G$ is called
the {\em pullup} of $P$.
The pullup functor $P^{\Cdot}$ has a left adjoint
$P\down\colon \Mod R \to \Mod R/G$, which is called
the {\em pushdown} of $P$.
Note that we have $P\down R(\blank, x)\iso R/G(\blank,P x)$
for all $x \in R$.
This together with the right exactness of $P\down$ shows
that $P\down$ induces a functor $P\down\colon \mod R \to \mod R/G$.
\item
The pullup $P^{\Cdot}$ and the pushdown $P\down$
induce functors $P^{\Cdot}\colon \calK(\Prj R/G) \to \calK(\Prj R)$ and
$P\down\colon\calK(\Prj R) \to \calK(\Prj R/G)$, respectively,
which also form an adjoint pair
$\xymatrix@1@C2ex{P\down \ar@{-|}[r] & P\up}$.
Note that $P\down$ also induces a functor
$P\down\colon\Kb(\prj R) \to \Kb(\prj R/G)$.
\item
Each $\al\in G$ defines an automorphism of $\Mod R$
by setting ${}^{\al}X:= X \circ A_{\al\inv}$ for all $X\in \Mod R$, by which
the $G$-action on $R$ induces a $G$-action on $\Mod R$.
Note that ${}^{\al}R(\blank,x) = R(\al\inv(\blank),x)\iso R(\blank, \al x)$
for all $x \in R$.

The $G$-action on $\Mod R$ canonically induces
that on $\calK(\Prj R)$ and on $\Kb(\prj R)$.
Namely, for each complex $X:=(X^i, d^i)_{i\in \bbZ}$ and $\al\in G$
set ${}^{\al}X:= ({}^{\al}X^i, {}^{\al}d^i)_{i\in \bbZ}$.
\end{enumerate}
\end{dfn}

\begin{thm}\label{pushdown-precov1}
Let $R$ be a
category, $G$ a group acting on $R$, and $P\colon R \to R/G$
the canonical $G$-covering.
Then the pushdown functor $P\down\colon \mod R \to \mod R/G$
is a $G$-precovering. 
\end{thm}

\begin{proof}
First of all we give the precise form of the pushdown
$P\down=(P\down, \ph\down)$ as a  $G$-invariant functor.

{\em Definition of $P\down$}:

On objects: For each $X \in \Mod R$
the module $P\down X \in \Mod R/G$ is defined as follows:

For each $x\in \obj(R/G)=\obj(R)$, $(P\down X)(x):= \Ds_{\al \in G} X(\al x)$;

for each $f\colon x \to y$ in $R/G$ with
$f=(f_{\be,\al})_{\al, \be \in G}\in (R/G)(x, y) \subseteq \prod_{\al, \be \in G}R(\al x, \be y)$,
$(P\down X)(f)$ is defined by the commutative diagram
\begin{equation}\label{pi-down-obj}
\begin{CD}
(P\down X)(y) @>(P\down X)(f)>> (P\down X)(x)\\
@\vert @\vert\\
\Ds_{\be\in G}X(\be y) @>>(X(f_{\be,\al}))_{\be,\al}> \Ds_{\al\in G}X(\al x).
\end{CD}
\end{equation}

On morphisms: For each morphism $u\colon X \to X'$ in $\Mod R$,
the morphism $P\down u \colon P\down X \to P\down X'$
is defined as follows:
$P\down u:= ((P\down u)_{x})_{x \in \obj(R/G)}$, where
for each $x \in \obj(R/G)$, $(P\down u)_{x}$ is defined by the commutative diagram
\begin{equation}\label{pi-down-mor}
\begin{CD}
(P\down X)(x) @>(P\down u)_{x}>> (P\down X')(x)\\
@\vert @\vert\\
\Ds_{\al\in G}X(\al x) @>>\Ds_{\al\in G}u_{\al x}> \Ds_{\al \in G}X'(\al x).
\end{CD}
\end{equation}
Then for each $f\colon x \to y$ in $R/G$ as above we have
a commutative diagram
$$
\begin{CD}
\Ds_{\be\in G}X(\be y) @>(X(f_{\be,\al}))_{\be,\al}>> \Ds_{\al\in G}X(\al x)\\
@V\Ds_{\be\in G}u_{\be y}VV @VV\Ds_{\al\in G}u_{\al x}V\\
\Ds_{\be\in G}X'(\be y) @>>(X'(f_{\be,\al}))_{\be,\al}> \Ds_{\al\in G}X'(\al x),
\end{CD}
$$
which shows that $P\down u$ is a morphism in $\Mod R/G$.
This defines a functor $P\down\colon \Mod R \to \Mod R/G$.
Then $P\down$ is a left adjoint to the pullup $P^{\Cdot} \colon \Mod R/G \to \Mod R$.
Indeed, for each $X \in \Mod R$ and $Y \in \Mod R/G$ the adjunction
$$
\th_{X,Y}\colon \Hom_{R/G}(P\down X, Y) \to \Hom_R(X,P\up Y)
$$
is given by
$(\th_{X,Y}t)_x:= t_{x,1}\colon X(x) \to Y(x)=Y(P x)=(P\up Y)(x)$
for each $x \in \obj(R)=\obj(R/G)$
and $t \in \Hom_{R/G}(P\down X, Y)$
with $t= (t_x)_{x \in R/G}$ and
$t_x = (t_{x,\al})_{\al\in G}\colon \Ds_{\al\in G}X(\al x)$ $\to Y(x)$;
and its inverse 
$$
\th\inv_{X,Y}\colon \Hom_R(X,P\up Y) \to \Hom_{R/G}(P\down X, Y)
$$
is given by $(\th\inv_{X,Y}f)_x:= (Y(\ph_{\al,x})f_{\al x})_{\al\in G}$
for each $f\in \Hom_R(X,P\up Y)$ and $x \in R/G$.

Here, note that by construction
$
(P\up P\down X)(x) = \Ds_{\al\in G}X(\al x)=(\Ds_{\al\in G}{}^{\al\inv}X)(x) \iso (\Ds_{\al\in G}{}^{\al}X)(x)
$
for all $X \in \Mod R$ and $x\in R$,
which yields the canonical isomorphism:
$$
P\up P\down X \iso \Ds_{\al\in G}{}^{\al}X.
$$

{\em Definition of $\ph\down$}:

For each $\mu \in G$ define a morphism
$\ph\down{}_{\mu}\colon P\down \to P\down\circ{}^{\mu}(\blank)$ by
$\ph\down{}_{\mu}:= (\ph\down{}_{\mu, X})_{X\in \Mod R}$, where for each
$X \in \Mod R$, the morphism $\ph\down{}_{\mu, X}$ is given by
$\ph\down{}_{\mu, X}:= (\ph\down{}_{\mu, X, x})_{x\in R}$ and
by the commutative diagram
$$
\begin{CD}
(P\down X)(x) @>\ph\down{}_{\mu, X, x}>> (P\down ({}^{\mu}X))(x)\\
@\vert @\vert\\
\Ds_{\al \in G}X(\al x) @>>(\de_{\al,\mu\inv \be}\id_{X\al x})_{\al,\be\in G}> \Ds_{\be\in G}X(\mu\inv\be x)
\end{CD}
$$
for each $x \in R$.
Then $\ph\down{}_{\mu}$ turns out to be a natural isomorphism
for each $\mu\in G$,
and the family $\ph\down:=(\ph\down{}_{\mu})_{\mu\in G}$ is easily verified to be an invariance adjuster.
Thus the pair $P\down=(P\down, \ph\down)$ is a  $G$-invariant
functor.

For each $X, Y \in \mod R$ using the description of
$(P\down, \ph\down)$ above,
it is not hard to check the commutativity of the following diagram
with canonical maps:

$$
\begin{CD}
\Ds_{\al\in G}(\mod R)(X, {}^{\al}Y) @>\sim >> 
  (\Mod R)(X, \Ds_{\al \in G}{}^{\al}Y)\\
@V{P^{(2)}_{\Cdot\ X,Y}}VV @VV\wr V\\
(\mod R/G)(P\down X, P\down Y)
  @>\sim >> (\Mod R)(X, P\up P\down Y),
\end{CD}
$$
which shows that $P\down=(P\down, \ph\down)$ is a $G$-precovering.
\end{proof}

\begin{thm}\label{pushdown-precov}
Let $R$ be a
category, $G$ a group acting on $R$, and $P\colon R \to R/G$
the canonical $G$-covering.
Then the pushdown functor $P\down\colon\Kb(\prj R) \to \Kb(\prj R/G)$
is a $G$-precovering. 
\end{thm}

\begin{proof}
Let $X, Y \in \Kb(\prj R)$.
Then since $X$ is compact, the canonical homomorphism
$\Ds_{\al\in G}\Kb(\prj R)(X, {}^{\al}Y) \to 
  \calK(\Prj R)(X, \Ds_{\al \in G}{}^{\al}Y)$
is an isomorphism.
The description of $P\down=(P\down, \ph\down)$ above canonically yields
that of the pushdown functor between homotopy categories.
Then the commutativity of
the diagram
$$
\begin{CD}
\Ds_{\al\in G}\Kb(\prj R)(X, {}^{\al}Y) @>\sim >> 
  \calK(\Prj R)(X, \Ds_{\al \in G}{}^{\al}Y)\\
@V{P^{(2)}_{\Cdot\ X,Y}}VV @VV\wr V\\
\Kb(\prj R/G)(P\down X, P\down Y)
  @>\sim >> \calK(\Prj R)(X, P\up P\down Y)
\end{CD}
$$
with canonical maps follows from that of the diagram in the proof of the previous theorem, and the theorem is proved.
\end{proof}

To state the next result we need some terminologies.
\begin{dfn}
Let $R$ be a category and $G$ a group.
\begin{enumerate}
\item
A full subcategory $E$ of $\Kb(\prj R)$ is called a {\em tilting subcategory}
for $R$ if it has the following properties:
\begin{enumerate}
\item
$\Kb(\prj R)(U, V[i]) = 0$ for all $U,V \in E$ and for all $i\ne 0$;
\item
$R(\blank, x) \in \thick E$ for all $x \in R$, where $\thick E$ is
the {\em thick} subcategory generated by $E$, i.e., the smallest full triangulated subcategory of $\Kb(\prj R)$
containing $E$ closed under isomorphisms and direct summands.
\end{enumerate}
\item
Assume that $R$ has a $G$-action.
A tilting subcategory $E$ of $\Kb(\prj R)$ is called {\em $G$-stable}
if ${}^{\al}U \in E$ for all $U\in E$ and $\al \in G$.
\item
Two categories $R$ and $S$ are said to be
{\em derived equivalent}
if the derived categories $\calD(\Mod R)$ and $\calD(\Mod S)$
are equivalent as triangulated categories.
\end{enumerate}
\end{dfn}

To apply the following theorem
we assume throughout the rest of this section
except for Definition \ref{dfn:equivariant},
Remark \ref{rmk:equivariant} and Lemma \ref{equivariant-eq} that
the categories $R$ in consideration are 
{\em small and $\k$-flat}, in the sense that
$R(x,y)$ is a flat $\k$-module for each $x, y \in R$.
(When $R$ is a differential graded category as in \cite{Ke:DerDGcat},
the definition of $\k$-flatness should be slightly changed, but
in the usual category case the definition above works.)

By Rickard \cite{Ric} and Keller \cite[9.2, Corollary]{Ke:DerDGcat}
the following is known.

\begin{thm}
\label{der-eq}
Two categories $R$ and $S$ are
derived equivalent
if and only if
there exists a tilting subcategory $E$ for $R$ such that
$E$ is equivalent to $S$.
\end{thm}

The following is a fundamental theorem of covering technique for
derived equivalence.

\begin{thm}
\label{fundamental-thm}
Let $G$ be a group and $R$ a category with a $G$-action
$($not necessarily a free action$)$.
Assume that there exists a $G$-stable tilting subcategory
$E$ for $R$.
Then $R/G$ and $E/G$ are derived equivalent.
\end{thm}

\begin{proof}
Set $E'$ to be the full subcategory of $\Kb(\prj R/G)$
consisting of the objects $P_{\Cdot}U$ with $U \in E$.
By \thmref{der-eq} we have only to show that
$E'$ is a tilting subcategory for $R/G$ and that
$E'$ is equivalent to $E/G$.
Now for each $U, V \in E$ and for each integer $i \ne 0$
Theorem \ref{pushdown-precov} shows that
$\Kb(\prj R/G)(P\down U,P\down V[i])
\iso \Ds_{\al\in G}\Kb(\prj R)({}^{\al}U,V[i])=0$
because ${}^{\al}U \in E$.
Next for each $x\in R/G$ we have
$(R/G)(\blank, x) \iso P_{\Cdot}(R(\blank, x)) \in P_{\Cdot}(\thick E)
\subseteq \thick E'$.
Therefore $E'$ is a tilting subcategory for $R/G$.
Finally, since the restriction of
$P_{\Cdot}\colon \Kb(\prj R) \to \Kb(\prj R/G)$ to $E$ induces
a $G$-precovering $E \to E'$ that is dense,
$E'$ is equivalent to $E/G$ by \thmref{chrtrz-cov}.
\end{proof}

\begin{dfn}\label{dfn:equivariant}
Let $\calC, \calC'$ be categories with $G$-actions
and $F\colon \calC \to \calC'$ a functor.
Then an {\em equivariance adjuster} of $F$ is a family
$\ro = (\ro_{\al})_{\al\in G}$ of natural isomorphisms
$\ro_{\al}\colon A_{\al}F \to FA_{\al}$ ($\al\in G$)
such that the following diagram commutes for each $\al.\be\in G$
$$
\xymatrix{
A_{\be\al}F=A_{\be}A_{\al}F \ar[r]^{A_{\be}\ro_{\al}}\ar[rd]_{\ro_{\be\al,x}}
& A_{\be}FA_{\al} \ar[d]^{\ro_{\be,\al}}\\
&FA_{\be\al}=FA_{\be}A_{\al},
}
$$
and the pair $(F,\ro)$ is called a $G$-{\em equivariant}
functor.
In particular, $F$ is called a {\em strictly} $G$-{\em equivariant}
functor if the $\ro_{\al}$ above can be taken to be the identity, i.e.,
if $A_{\al}F=FA_{\al}$ for all $\al\in G$.
\end{dfn}

\begin{rmk}\label{rmk:equivariant}
In the setting of Definition \ref{dfn:equivariant} 
let $P=(P,\ph)\colon \calC' \to \calC'/G$ be
the canonical $G$-covering functor.
For each $\al \in G$ define a natural isomorphism
$\ph'_{\al}\colon PF \to PFA_{\al}$ by
$$
\ph'_{\al}:= (P\ro_{\al})\circ (\ph_{\al}F)
\colon PF \to PA_{\al}F \to P FA_{\al},
$$
and set $\ph':=(\ph'_{\al})_{\al\in G}$.
Then a direct calculation shows that
$\ph'$ is an invariance adjuster if and only if
$\ro$ is an equivariance adjuster.
\end{rmk}

\begin{lem}
\label{equivariant-eq}
Let $\calC, \calC'$ be categories with $G$-actions, and
$(F, \ro) \colon \calC \to \calC'$ a
$G$-equivariant equivalence.
Then $\calC/G$ and $\calC'/G$ are equivalent.
\end{lem}

\begin{proof}
Let $P=(P,\ph)\colon \calC' \to \calC'/G$ be
the canonical $G$-covering functor.
Define a family $\ph'=(\ph'_{\al})_{\al\in G}$ of natural isomorphisms
$\ph'_{\al}\colon PF \to PF A_{\al}$ ($\al \in G$)
as in Remark \ref{rmk:equivariant} above.
Then as stated there $\ph'$ is an invariance adjuster
and the pair $(PF, \ph')$
becomes a  $G$-invariant functor $\calC \to \calC'/G$.
We show that it is a $G$-covering functor.
First, since $F$ is an equivalence, $PF$ is dense.
Next, by the definition of $\ph'$ we have
the following commutative diagram:
$$
\xymatrix@C15ex{
\Ds_{\al\in G}\calC(\al x, y) \ar[r]^{(PF)^{(1)}_{x,y}}
\ar[d]_{\Ds_{\al\in G}F_{\al x,y}}
&\calC'/G(PFx, PFy)\\
\Ds_{\al\in G}\calC'(F\al x,Fy)\ar[r]_{\Ds_{\al\in G}\calC'(\ro_{\al,x},F y)}&
\Ds_{\al\in G}\calC'(\al Fx,Fy),
\ar[u]_{P^{(1)}_{Fx,Fy}}
}
$$
where the vertical morphisms and the bottom morphism
are isomorphisms by assumptions,
which shows that $(PF, \ph')$ is a $G$-precovering.
Thus $(PF, \ph')$ turns out to be a $G$-covering.
Hence $\calC/G$ and $\calC'/G$ are equivalent
by Theorem \ref{chrtrz-cov}.
\end{proof}

In applications of this section
we usually deal with the case that $E=\calC$ and $S=\calC'$ are
basic categories and $\ps=F$ is a strictly $G$-equivariant isomorphism
between them.
The notion of $G$-equivariant functors plays an essential role
in section 5.

\begin{thm}\label{thm:orbit-cat}
Let $G$ be a group and $R, S$
categories with
$G$-actions $($not necessarily free actions$)$.
Assume that there exists a $G$-stable tilting subcategory
$E$ for $R$ and a $G$-equivariant equivalence
$E \to S$.
Then $R/G$ and $S/G$ are derived equivalent.
\end{thm}

\begin{proof}
This follows from Theorem \ref{fundamental-thm} and
Lemma \ref{equivariant-eq}.
\end{proof}

This together with the remark in Definition \ref{dfn:skew-gp-cat} shows
the following.
\begin{cor}\label{thm:skewgp-cat}
Let $G$ be a group and $R, S$ categories with
$G$-actions $($not necessarily free actions$)$.
Assume that there exists a $G$-stable tilting subcategory
$E$ for $R$ and a $G$-equivariant equivalence
$E \to S$.
Then $R*G$ and $S*G$ are derived equivalent.
\end{cor}

\section{Smash products and orbit categories}

In this section we generalize a result in \cite{C-M}  giving a relationship
between smash products and orbit categories,
namely we prove it without an assumption that the $G$-action is free.

We cite the following two definitions from \cite{C-M}.

\begin{dfn}
(1) A $G$-{\em graded} category is a category $\calB$ having a family of
direct sum decompositions
$$\calB(x,y) = \Ds_{\al\in G}\calB^{\al}(x,y)$$
($x, y \in \calB$) of $\k$-modules such that
the composition of morphisms gives the inclusions
$\calB^{\be}(y, z) \cdot \calB^{\al}(x,y) \subseteq \calB^{\be\al}(x,z)$
for all $x,y,z\in \calB$ and $\al, \be \in G$.

(2) For each $f\in \calB(x,y)$ we set $\deg f:= \al$ if $f\in \calB^{\al}(x,y)$
for some $\al \in G$ (obviously such an $\al$ is uniquely determined
by $f$ if it exists).

(3) A functor $H\colon \calB \to \calB'$ of $G$-graded categories
is called {\em degree-preserving} if
$H(\calB^{\al}(x, y)) \subseteq \calB'^{\al}(Hx,Hy)$
for all $x, y \in \calB$ and $\al\in G$.
\end{dfn}

\begin{dfn}
Let $\calB$ be a $G$-graded category.
Then the smash product $\calB \# G$ of $\calB$ and $G$ is a category defined as follows:
\begin{itemize}
\item
$\obj(\calB \# G):= \obj(\calB) \times G$ (we set $x^{(\al)}:=(x,\al)$ for all $(x,\al) \in \obj(\calB \# G)$);
\item
$(\calB \#G)(x^{(\al)}, y^{(\be)}):= \calB^{\be\inv\al}(x, y)$
for each $x^{(\al)}, y^{(\be)} \in \obj(\calB \# G)$; and
\item
The composition
$$(\calB \# G)(y^{(\be)}, z^{(\ga)}) \times (\calB \# G)(x^{(\al)},y^{(\be)}) \to (\calB \# G)(x^{(\al)},z^{(\ga)})$$
is given by the composition
$$\calB^{\ga\inv\be}(y, z) \times \calB^{\be\inv\al}(x,y) \to \calB^{\ga\inv\al}(x,z)$$
of $\calB$
for each $x^{(\al)}, y^{(\be)}, z^{(\ga)} \in \obj(\calB \# G)$.
\end{itemize}
\end{dfn}

\begin{rmk}
When $\calB =R$ is a graded algebra, i.e.,
a graded category with only one object, with a direct sum decomposition
$R = \Ds_{\al\in G}R^{\al}$, the algebra corresponding to the category
$R\# G$ is given by
$$
\ds(R\# G) = \Ds_{(\al,\be)\in G\times G}R^{\be\inv\al}
$$
with the usual matrix multiplication (see the formula \eqref{cat-alg}).
When $G$ is a finite group, this coincides with the smash product
construction given by Quinn \cite{Qu}.
\end{rmk}

\begin{lem}
Let $\calC$ be a category with a $G$-action.
Then $\calC/G$ is $G$-graded.
\end{lem}

\begin{proof}
Let $P\colon \calC \to \calC/G$ be the canonical $G$-covering functor.
For each $x, y \in \obj(\calC)=\obj(\calC/G)$ we have an isomorphism
$P^{(1)}_{x,y}\colon \Ds_{\al\in G}\calC(\al x, y) \to (\calC/G)(P x, P y)$ having $\si^{(1)}_{x,y}$ as the
inverse.
Therefore by setting $(\calC/G)^{\al}(x,y):=P^{(1)}_{x,y}(\calC(\al x, y))$ for all $\al \in G$,
we have $(\calC/G)(x,y) = \Ds_{\al\in G}(\calC/G)^{\al}(x,y)$.
As easily seen $\calC/G$ together with these decompositions turns out to be a $G$-graded category.
\end{proof}

\begin{rmk}\label{degree}
In the lemma above,
let $\be\in G$ and $f\in \calC/G(P y, P x)$ with $x,y\in \calC$.
Then $f\in (\calC/G)^{\be}(P y, P x)$ (i.e.~$\deg f = \be$)
if and only if
$f_{\mu,\la} = \de_{\mu\inv\la,\be}f_{\mu,\la}$ for all $\la,\mu\in G$.

Indeed, 
\begin{align*}
f\in (\calC/G)^{\be}(P y, P x)
&\iff \si^{(1)}_{y,x}(f) =(f_{1,\la})_{\la\in G} \in \calC(\be y, x) \subseteq \Ds_{\la\in G}\calC(\la y, x)\\
&\iff \forall \la\in G, f_{1,\la} = \de_{\la,\be}f_{1,\la} \\
&\iff \forall \la,\mu \in G, f_{\mu,\la} = \mu(f_{1,\mu\inv\la})
=\mu(\de_{\mu\inv\la,\be}f_{1,\mu\inv\la})=\de_{\mu\inv\la,\be}f_{\mu,\la}.
\end{align*}
\qed
\end{rmk}

A statement similar to the following is stated
in \cite[Proposition 3.2]{C-M}.
We give a proof that does not use their orbit category $\calC\freeorbit G$.
(Since by our definition $\obj(\calC/G) = \obj(\calC)$ for a category $\calC$ with
a $G$-action, we cannot state that
$(\calB \# G)/G$ is {\em isomorphic} to $\calB$ in general.)

\begin{prp}
\label{smash-orb}
Let $\calB$ be a $G$-graded category.
Then the smash product $\calB \# G$ is a category with a free $G$-action,
and there is a degree-preserving equivalence $\calB \to (\calB \# G)/G$ of
$G$-graded categories.
\end{prp}

\begin{proof}
First we give a $G$-action on $\calB \# G$.

On objects: $\mu x^{(\al)}:= x^{(\mu\al)}$ for each $\mu \in G$ and each $x^{(\al)}\in \obj(\calB\# G)$;

On morphisms: $\mu f:= f$ for each $\mu \in G$ and
each $f\in (\calB\# G)(x^{(\al)}, y^{(\be)})=\calB^{\be\inv\al}(x,y) = (\calB\# G)(\mu x^{(\al)}, \mu y^{(\be)})$
with $x^{(\al)}, y^{(\be)} \in \obj(\calB\# G)$.

Then it is easy to verify that the action of each $\mu \in G$ defined above is an automorphism of
the category $\calB\# G$ and that this defines a $G$-action on $\calB\# G$.
This $G$-action is free because $\mu x^{(\al)} = x^{(\al)}$ implies $\mu\al=\al$ and $\mu=1$.
With this free $G$-action we consider $(\calB\# G)/G$.

Let $(P, \ph)\colon \calB\# G \to (\calB\# G)/G$
be the canonical $G$-covering functor.
We define a functor $\om_{\calB}\colon \calB \to (\calB\# G)/G$ as follows.

On objects: $\om_{\calB}x:= P(x^{(1)})$ for each $x\in \calB$; and

On morphisms: $\om_{\calB}f:= P^{(1)}_{x^{(1)},y^{(1)}}(f)$
for all $x, y \in \calB$ and $f \in \calB(x,y)$.
Using the definition of $G$-action on $\calB\# G$,
it is easy to verify that $\om_{\calB}$ is a functor.
Since the isomorphism $P^{(1)}_{x^{(1)},y^{(1)}}$ sends
$\calB(x,y)=\Ds_{\al\in G}\calB^{\al}(x,y)
=\Ds_{\al\in G}(\calB\# G)(x^{(\al)}, y^{(1)})
=\Ds_{\al\in G}(\calB\# G)(\al x^{(1)}, y^{(1)})$
onto
$((\calB\# G)/G)(P(x^{(1)}), P(y^{(1)})) = ((\calB\# G)/G)(\om_{\calB}x, \om_{\calB}y)$,
and each $\calB^{\al}(x,y)$ onto $((\calB\# G)/G)^{\al}(\om_{\calB}x, \om_{\calB}y)$,
$\om_{\calB}$ is fully faithful and degree-preserving.
Finally, for each $P(x^{(\al)}) \in \obj((\calB\# G)/G)$ (with $x^{(\al)} \in \calB\# G$), 
we have $P(x^{(\al)}) = P(\al x^{(1)}) \iso P(x^{(1)})= \om_{\calB}x$ in $(\calB\# G)/G$.
Thus $\om_{\calB}$ is dense.
As a consequence, $\om_{\calB}$ is a degree-preserving equivalence of $G$-graded categories.
\end{proof}

\begin{dfn}
Let $\calB$ be a $G$-graded category.
Then we define a functor $Q\colon \calB\# G \to \calB$ as follows.

On objects: $Q x^{(\al)}:= x$ for all $x^{(\al)}\in \calB\# G$.

On morphisms: $Q f:= f$ for all $f\in (\calB\# G)(x^{(\al)},y^{(\be)})
= \calB^{\be\inv\al}(x,y) \subseteq \calB(x,y)$ and for all
$x^{(\al)},y^{(\be)} \in \calB\# G$.
\end{dfn}

\begin{prp}\label{G-covering-assoc-smash}
$Q = QA_{\al}$ for all $\al\in G$ and $Q=(Q, \id)\colon \calB\# G \to \calB$ is
a $G$-covering functor,
where $\id$ denotes the invariance adjuster $\id = (\id_{Q}:Q \to QA_{\al})_{\al\in G}$.
\end{prp}

\begin{proof}
Straightforward.
\end{proof}

\begin{rmk}
Let $(P,\ph)\colon \calB\# G \to (\calB\# G)/G$ be the canonical
$G$-covering functor.
Then by Proposition \ref{G-covering-assoc-smash} and Theorem \ref{chrtrz-cov}
we have an equivalence
$H\colon (\calB\# G)/G \to \calB$ such that $(Q,\id) = (HP,H\ph)$.
But $H$ is not degree-preserving in general.
Indeed, we have
$$
H(((\calB\# G)/G)^{\mu}(x^{(\al)},y^{(\be)}))
\subseteq \calB^{\be\inv\mu\al}(x,y)
$$
for all $\mu \in G$ and $x^{(\al)},y^{(\be)}\in (\calB\# G)/G$.
\end{rmk}

The following is a generalization of \cite[Theorem 3.8]{C-M}.

\begin{thm}
\label{liberalization}
Let $\calC$ be a category with a $G$-action $($not necessarily a free action$)$.
Then there is a $G$-equivariant equivalence
$\calC \to (\calC/G)\# G$.
\end{thm}

\begin{proof}
Let $P\colon \calC \to \calC/G$ be the canonical $G$-covering functor.
We define a functor $\ep_{\calC}\colon \calC \to (\calC/G)\# G$
as follows.

On objects: $\ep_{\calC}x:= x^{(1)}$ for each $x \in \calC$.

On morphisms: $\ep_{\calC}f:= P^{(1)}_{x,y}(f)$ for each $x,y\in \calC$ and each $f\in \calC(x,y)$.
Note that $P^{(1)}_{x,y}\colon \calC(x,y) \to (\calC/G)^{1}(x,y) = ((\calC/G)\# G)(x^{(1)},y^{(1)})$
is an isomorphism.
As easily seen $\ep_{\calC}$ is a functor.
By construction it is obvious that $\ep_{\calC}$ is fully faithful.
We show that $\ep_{\calC}$ is dense.
For this it is enough to show that
$x^{(\al)} \iso (\al x)^{(1)}$ in $(\calC/G)\# G$
for each $x^{(\al)}\in \obj((\calC/G)\# G)$ (with $x \in \obj(\calC/G), \al \in G$)
because $(\al x)^{(1)} = \ep_{\calC}(\al x)$.
Since $((\calC/G)\# G)((\al x)^{(1)}, x^{(\al)}) = (\calC/G)^{\al\inv}(\al x, x) =
P^{(1)}_{\al x, x}\calC(\al\inv \al x,x) \ni P^{(1)}_{\al x,x}(\id_{x}) =\ph_{\al\inv,\al x}$,
there is a morphism $\ph_{\al\inv,\al x} \colon (\al x)^{(1)}\to x^{(\al)}$ in $(\calC/G)\# G$.
Further  since
$((\calC/G)\# G)(x^{(\al)}, (\al x)^{(1)}) = (\calC/G)^{\al}(x, \al x) = P^{(1)}_{x,\al x}\calC(\al x, \al x)
\ni P^{(1)}_{x, \al x}(\id_{\al x}) = \ph_{\al,x}$, we have a morphism
$\ph_{\al, x}\colon x^{(\al)} \to (\al x)^{(1)}$ in $(\calC/G)\# G$.
These morphisms $\ph_{\al,x}$ and $\ph_{\al\inv,\al x}$ are inverse to each other also in $(\calC/G)\# G$
by Remark \ref{superfluous-cond} because as easily seen we have $\id_{x^{(1)}}= \id_{P x}$
and $\id_{(\al x)^{(1)}} = \id_{P \al x}$.
Hence $x^{(\al)} \iso (\al x)^{(1)} = \ep_{\calC}(\al x)$ in $(\calC/G)\# G$.
As a consequence, $\ep_{\calC}$ is an equivalence.

Finally, we make $\ep_{\calC}$ a $G$-equivariant functor.
Define a family $\ro = (\ro_{\al})_{\al\in G}$ of natural transformations
$\ro_{\al} \colon A_{\al} \ep_{\calC} \to \ep_{\calC}A_{\al}$ ($\al \in G$) by
$\ro_{\al,x}:= \ph_{\al, x} \colon \al \ep_{\calC}x = x^{(\al)} \to 
(\al x)^{(1)} = \ep_{\calC}\al x$ for all $x \in \calC$.
Then for each $\al \in G$, $\ro_{\al}$ is a natural isomorphism
because $\ph_{\al}\colon P \to PA_{\al}$ is.
To show that $\ro$ is an equivariance adjuster
it is enough to verify that
$\ro_{\be,\al x}\cdot \be\ro_{\al,x} = \ro_{\be\al, x}$
for each $\al, \be \in G$.
Noting that $\be\ro_{\al, x} = \be\ph_{\al, x} = \ph_{\al,x}$ by the definition of $G$-action on $(\calC/G)\# G$,
we see that this follows from the fact that
$\ph$ is an invariance adjuster.
\end{proof}

\begin{rmk}
(1) In the above theorem, note that the $G$-action
on the right hand side is free, whereas on the left hand side
it is not always free.
Thus by passing from $\calC$ to $(\calC/G)\# G$
we can change any $G$-action to a free $G$-action that is
equivariant to the original one.
See Example \ref{exm:liberalization}
for an example of this ``liberalization''.

(2) Proposition \ref{smash-orb} and Theorem \ref{liberalization}
give a full categorical generalization of Cohen-Montgomery duality
\cite{Co-Mo}.
\end{rmk}

\begin{dfn}
Let $F\colon \calC \to \calB$ be a dense functor.
Take a subclass $\calI$ of $\obj(\calC)$ such that
$F$ is injective on $\calI$ and
$\{Fu\mid u \in \calI \}$ forms a complete set of representatives of
$F(\obj(\calC))/\!\iso$.
Then for each $x\in \calB$ there is a unique $I_x \in \calI$ such that
there is an isomorphism $\nu_x\colon F(I_x) \to x$.
When $x \in F(\calI)$, we have $x = F(I_x)$ and in this case
we take $\nu_x:=\id_x$.
We call the pair $(I,\nu)$ of such families $I:=(I_x)_{x\in \calB}$
and $\nu:=(\nu_x)_{x\in \calB}$ an {\em essential section} of $F$.
\end{dfn}

\begin{lem}
Let $F\colon \calC \to \calB$ be a $G$-covering functor.
Then since $F$ is dense, there is an essential section
$(I,\nu)$ of $F$.
Using this we set
$$
\calB^{\al}(x,y):=
\nu_{y}F_{I_x,I_y}^{(1)}(\calC(\al I_x,I_y))\nu_{x}\inv
$$
for all $\al \in G$ and all $x,y\in \calB$.
Then this makes $\calB$ a $G$-graded category.
This $G$-grading of $\calB$ is called a $G$-{\em grading induced by} $F$
$($with respect to $(I,\nu)$$)$. 
\end{lem}

\begin{proof}
Straightforward.
\end{proof}

\begin{cor}
\label{covering-orbit-smash}
Let $F=(F,\ps)\colon \calC \to \calB$ be a $G$-covering functor and
$(I,\nu)$ an essential section of $F$.
Regard $\calB$ as a $G$-graded category by the $G$-grading
induced by $F$ with respect to $(I,\nu)$.
Then
there is a degree-preserving equivalence $I' \colon \calB \to \calC/G$
of $G$-graded categories such that the diagram
$$
\begin{CD}
\calC @>{(F,\ps)}>>\calB\\
@\vert @VV{I'}V \\
\calC @>>{(P,\ph)}> \calC/G
\end{CD}
$$
is commutative up to natural isomorphisms of  $G$-invariant
functors.
\end{cor}

\begin{proof}
By Theorem \ref{chrtrz-cov} there exists an equivalence
$H\colon \calC/G \to \calB$ such that $(F,\ps) = (HP, H\ph)$.
In particular, we have a commutative diagram:
\begin{equation}\label{comm:F-factors-th-pi}
\begin{CD}
\calC @>{(F,\ps)}>> \calB\\
@\vert @AA{H}A\\
\calC @>>{(P, \ph)}> \calC/G.
\end{CD}
\end{equation}
Now define a functor $I'\colon \calB \to \calC/G$ as follows.

On objects: $I'x:= I_x$ for all $x \in \calB$.

On morphisms:
$I'f:= H_{P I_x, P I_y}\inv (\nu_y\inv f\nu_x)$
for all $f \in \calB(x,y)$ and $x,y\in \calB$, where
$H_{P I_x, P I_y}\colon (\calC/G)(P I_x, P I_y) \to \calB(FI_x, FI_y)$ is an isomorphism that is a restriction of the equivalence
$H\colon \calC/G \to \calB$.

Then as easily seen $I'$ is well-defined as a functor.
By the definitions of the grading of $\calB$ and of $I'$
we have $I'(\calB^{\al}(x,y)) = (\calC/G)^{\al}(I'x, I'y)$
for all $x,y \in \calB$ and $\al \in G$.
Thus $I'$ is a degree-preserving functor.
By construction $I'$ is fully faithful.
Since $H$ is an equivalence,
there exists a unique isomorphism $\mu_x\colon I_{Fx} \to P x$
in $\calC/G$ such that $H(\mu_x) = \nu_{Fx}$ for all $x \in \calC$.
In particular, this shows that $I'$ is dense.
It is easy to see that $\mu:=(\mu_x)_{x\in \calC}$ is
a natural isomorphism $I'F \to P$.
Moreover it follows from the definition of $I'$ and 
the commutativity of \eqref{comm:F-factors-th-pi} that
$\mu\colon (I'F, I'\ps) \to (P,\ph)$ is a morphism
of  $G$-invariant functors.
%
%
%
%
\end{proof}

\section{Relationship between $\Mod \calC$ and $\Mod \calC/G$}

As before let $G$ be a group and $\calC$ a category with a $G$-action.
Let $P=(P, \ph)\colon \calC \to \calC/G$ be
the canonical $G$-covering functor.
In this section we show that the pullup functor $P\up\colon \Mod\calC/G \to \Mod\calC$
induces an equivalence
between $\Mod \calC/G$ and the category $\Mod^{G} \calC$
of ``$G$-invariant modules'' (see below for the definition),
and the pushdown functor $P\down\colon \Mod\calC \to \Mod\calC/G$
induces an equivalence between
$\Mod \calC$ and the category $\Mod_{G} \calC/G$
of $G$-graded modules and degree-preserving morphisms
(see below for the definitions).
Similar results were given in  \cite{C-M} (Theorems 4.3 and 4.5)
under the assumption that the $G$-action is free.
Here we do not assume this condition, and thus our results
give generalizations of these theorems in \cite{C-M}.

\begin{dfn}
\label{G-inv}
We set $\Mod^{G}\calC:= \Inv(\calC\op, \Mod \k)$,
and an object $(M, \ph)$ of which is called a {\em $G$-invariant} $\calC$-module.
There is a forgetful functor $\Mod^{G}\calC \to \Mod\calC$
sending $(M, \ph)$ to $M$, but usually it is
not injective on objects as shown in Example \ref{adjuster-representation}.
Note, therefore, that $\Mod^{G}\calC$ does not need to be a subcategory of $\Mod\calC$.
\end{dfn}

\begin{thm}\label{pullup-iso}
The pullup functor $P\up\colon \Mod \calC/G \to \Mod \calC$
induces an {\em isomorphism} $\Mod \calC/G \to \Mod^{G}\calC$
of categories.
\end{thm}

\begin{proof}
This follows by applying the contravariant version of
Corollary \ref{2-universality} to $\calC':= \Mod \k$
.
\end{proof}

\begin{rmk}
In the theorem above assume that the $G$-action on $\calC$ is free.
Then the same argument shows that
the classical canonical functor
$(P, \id)\colon \calC \to \calC\oorbit G$
induces an isomorphism $\Mod \calC\oorbit G \to \Mod^G\calC$,
which is exactly the content of Theorem 4.3 in \cite{C-M}.
In this case $P\up(M) = (M, \id)$ for all $M \in \Mod\calC/G$, and
$\Mod^G\calC$ is identified with the (right module version of)
category $(\calC\text{-}\Mod)^G$
defined in \cite{C-M}.
In particular, $\Mod^G\calC$ can be regarded as a full subcategory of $\Mod\calC$. 
\end{rmk}

In particular, if $\calC =R$ is an algebra we obtain the following.
\begin{cor}
Let $R$ be an algebra with a $G$-action.
Then we have an isomorphism $\Mod R*G \cong \Mod^{G}R$.\hfil \qed
\end{cor}


\begin{dfn}
\label{graded-mod-hmg}
Let $\calB$ be a $G$-graded category.
A $G$-{\em graded} $\calB$-module is a pair $(M, d)$ of
a $\calB$-module $M$ and
a family $d$ of direct sum decompositions
$M(x) = \Ds_{\al \in G} M^{\al}(x)\ (x \in \calB)$ such
that $M(f)(M^{\al}(x)) \subseteq M^{\al\be}(y)$
for all $f\in \calB^{\be}(y,x)$, $x,y\in \calB$ and $\be \in G$.
Let $M, N$ be $G$-graded $\calB$-modules and $u\colon M \to N$
a morphism between them as $\calB$-modules.
Then $u$ is called {\em degree-preserving} if
$u_{x}M^{\al}(x) \subseteq N^{\al}(x)$ for all $x\in \calB$ and $\al\in G$.
The category of $G$-graded modules and
degree-preserving morphisms between them is denoted by $\Mod_{G}\calB$.
Again there is a a forgetful functor $\Fgt\colon \Mod_{G}\calB \to \Mod\calB$
sending $(M,d)$ to $M$, but
usually it is not injective on objects, and $\Mod_{G}\calB$ does not
need to be a subcategory of $\Mod\calB$.

We denote by $\mod_G\calB$ the full subcategory of $\Mod_G\calB$
consisting of those $(M, d) \in \Mod_G\calB$ with $M \in \mod \calB$.
\end{dfn}

\begin{thm}\label{remov-grading}
The pushdown functor $P\down\colon \Mod \calC \to \Mod \calC/G$
induces an equivalence $\Mod \calC \to \Mod_{G}\calC/G$.
Namely, $P\down$ factors through the forgetful functor
to have a commutative diagram
$$
\xymatrix{
\Mod \calC && \Mod \calC/G\\
  & \Mod_{G}\calC/G,
\ar^{P\down}"1,1";"1,3"
\ar_{P'\down}"1,1"; "2,2"
\ar_{\Fgt}"2,2";"1,3"
}
$$
with $P'\down$ an equivalence.
\end{thm}

\begin{proof}
First we show that $P\down$ sends each $u\colon X \to X'$ in $\Mod \calC$
into $\Mod_{G} \calC/G$.
For each $x \in \calC$  we have
$(P\down X)(P x) = \Ds_{\al\in G} X(\al x)
$ by \eqref{pi-down-obj}.
Using this we set
\begin{equation}\label{grading-pi-down}
(P\down X)^{\al}(P x):= X(\al x).
\end{equation}
Then this makes $P\down X$ a $G$-graded $\calC/G$-module.
Indeed, for each
$f=(\de_{\mu\inv\la,\be}f_{\mu,\la})_{(\la,\mu)}
\in (\calC/G)^{\be}(P y, P x)$
with $x, y \in \calC$ and $\be \in G$
and for each
$a = (\de_{\mu,\al}a_{\mu})_{\mu} \in (P\down X)^{\al}(P x)$
with $\al\in G$,
we have
\begin{align*}
(P\down X)(f)(a)
&=(X(\de_{\mu\inv\la,\be}f_{\mu,\la}))_{(\mu,\la)}
(\de_{\mu,\al}a_{\mu})_{\mu}
= \left(\sum_{\mu\in G}\de_{\mu\inv\la,\be}X(f_{\mu,\la})(\de_{\mu,\al}a_{\mu})\right)_{\la}\\
&=(\de_{\la,\al\be}X(f_{\al,\la})(a_{\al}))_{\la}\in (P\down X)^{\al\be}(P y).
\end{align*}
Hence $P\down X \in \Mod_{G}\calC/G$.
Further the diagram \eqref{pi-down-mor}
shows that
$P\down u$ is degree-preserving, and $P\down u\colon P\down X \to P\down Y$ is in $\Mod_{G}\calC/G$, as desired.
Accordingly, the pushdown functor $P\down$ induces a functor
$P'\down\colon \Mod\calC \to \Mod_{G}\calC/G$, which we denote also by $P\down$.

We next show that this functor
is fully faithful.
The faithfulness is obvious by \eqref{pi-down-mor}.
To show that this functor is full, let
$X, Y \in \calC$ and $g \in \Mod_{G}\calC/G(P\down X, P\down Y)$.
Then for each $P x \in \calC/G$
we have
$g_{P x}\colon \Ds_{\al\in G}X(\al x) \to \Ds_{\al \in G}Y(\al x)$,
and
$g_{P x} = \Ds_{\al\in G} g_{\al, x}$ for some
$g_{\al, x} \colon X(\al x) \to Y(\al x)$ in $\Mod \k$
that is uniquely determined by $g_{P x}$ for each $\al\in G$.
Define a morphism $f\colon X \to Y$ in $\Mod \calC$ by
$f_x:= g_{1,x}\colon X(x) \to Y(x)$ for each $x \in \calC$.
\begin{clm}
$f$ is a morphism in $\Mod \calC$.
\end{clm}
Indeed, it is enough to show the commutativity of the diagram
$$
\begin{CD}
X(x) @>f_x>> Y(x)\\
@VX(h)VV @VVY(h)V\\
X(y) @>>f_y> Y(y)
\end{CD}
$$
for all $h\colon x \to y$ in $\calC$.
Noting that $(P\down X)(P h)=(X(\de_{\be,\al}\al h))_{(\al,\be)}=\Ds_{\al\in G}X(\al h)$,
this follows from the following commutative diagram
expressing that $g$ is in $\Mod_G \calC/G$:
$$
\begin{CD}
\Ds_{\al\in G}X(\al x) @>{\Ds_{\al\in G}g_{\al,x}}>> \Ds_{\al\in G}Y(\al x)\\
@V{\Ds_{\al\in G}X(\al h)}VV @VV{\Ds_{\al\in G}Y(\al h)}V\\
\Ds_{\al\in G}X(\al y) @>>{\Ds_{\al\in G}g_{\al,y}}> \Ds_{\al\in G}Y(\al y).
\end{CD}
$$
\begin{clm}
$P\down f = g$.
\end{clm}
Indeed, this is equivalent to saying that
$(P\down f)_{P x} = g_{P x}$
for all $P x \in \calC/G$, i.e., that
$\Ds_{\al \in G}f_{\al x} = \Ds_{\al\in G}g_{\al,x}$.
Hence it is enough to show the following for each $x\in \calC$ and $\al\in G$:
\begin{equation}
\label{g-one-al-x}
g_{1,\al x} = g_{\al,x}.
\end{equation}
Now since the isomorphism
$\ph_{\al,x}\colon P x \to P\al x$ in $\calC/G$ has the form
$\ph_{\al,x}=(\de_{\la,\mu\al}\id_{\la x})_{(\la,\mu)}$,
we have a commutative diagram
$$
\begin{CD}
P\down X(P x) @>{P\down X(\ph_{\al,x})}>> P\down X(P\al x)\\
\Vert &&\Vert\\
\Ds_{\la\in G}X(\la x) @>>{(X(\de_{\la,\mu\al}\id_{\la x}))_{\la,\mu}}> \Ds_{\mu\in G}X(\mu\al x).
\end{CD}
$$
Therefore $g\in \Mod \calC/G(P\down X, P\down Y)$ yields
a commutative diagram
$$
\begin{CD}
\Ds_{\la\in G} X(\la x) @>{\Ds_{\la\in G}g_{\la,x}}>> \Ds_{\la\in G}Y(\la x)\\
@V{(X(\de_{\la,\mu\al}\id_{\la x}))_{\la,\mu}}VV @VV{(Y(\de_{\la,\mu\al}\id_{\la x}))_{\la,\mu}}V\\
\Ds_{\mu\in G} X(\mu\al x) @>{\Ds_{\mu\in G}g_{\mu,\al x}}>> \Ds_{\mu\in G}Y(\mu\al x).
\end{CD}
$$
By a direct calculation this gives us the equality
$$
\de_{\la,\nu\al}g_{\la,x} = \de_{\la,\nu\al}g_{\nu,\al x}
$$
for all $\la,\nu\in G$.
In particular, for $\nu=1$ and $\la=\al$ we obtain the desired equation
\eqref{g-one-al-x}.
By these claims we see that the functor
$P\down\colon \Mod\calC \to \Mod_{G} \calC/G$ is full.

Finally, we show that this functor is dense.
Let $N=\Ds_{\al\in G}N^{\al}$ be in $\Mod_{G}\calC/G$.
We define an $M \in \Mod\calC$ as follows.

On objects: $M(x):= N^{1}(P x)= N^{1}(x)$ for all $x\in \calC$.

On morphisms: $M(f):= N(P f)|_{N^{1}(x)}\colon N^{1}(x) \to N^{1}(y)$
for all $f\colon x \to y$ in $\calC$.
\begin{clm}
$M$ is well-defined,
i.e., $N(P f)(N^1(x)) \subseteq N^1(y)$ for all $f\colon x \to y$ in $\calC$.
\end{clm}
Indeed, it is enough to show that
$\deg P f =1$.
But this follows from
$P f = (\de_{\be,\al}\al f)_{(\al,\be)}
=(\de_{\be\inv\al,1}\al f)_{(\al,\be)}$ by Remark \ref{degree}.

Next we show the following, which finishes the proof:
\begin{clm}
$P\down M \iso N$ in $\Mod_{G}\calC/G$.
\end{clm}
First note that for each $\al\in G$ and $x \in \calC$ we have
$$\deg \ph_{\al,x} =\al$$
by Remark \ref{degree} because
$
\ph_{\al,x}
= (\de_{\la,\mu\al}\id_{\la x})_{(\la,\mu)}
= (\de_{\mu\inv\la,\al}\id_{\la x})_{(\la,\mu)}
$. Then also $\deg \ph_{\al\inv,\al x} = \al\inv$ and hence
the mutually inverse isomorphisms $\ph_{\al,x}$ and $\ph_{\al\inv,\al x}$
induce the isomorphism
$$F_{\al,P x}:=
N(\ph_{\al,x})|_{N^1(P\al x)}\colon M(\al x)=N^1(P\al x) \to N^{\al}(P x)
$$
with the inverse $N(\ph_{\al\inv\al x})|_{N^{\al}(P x)}$ in $\Mod \k$.
Using this define an isomorphism $F_{P x}$ in $\Mod \k$
for each $x \in \calC$ by 
the commutative diagram
$$
\begin{CD}
(P\down M)(P x) @>{F_{P x}}>> N(P x)\\
\Vert && \Vert\\
\Ds_{\al\in G} M(\al x)
@>>{\Ds_{\al\in G}F_{\al,P x}}>
\Ds_{\al\in G} N^{\al}(P x).
\end{CD}
$$
To show this claim it is enough to show that
$F:=(F_{P x})_{P x \in \calC_G}\in \Mod_{G}\calC/G(P\down M, N)$.
To this end it is enough to show
that $F$ is in $\Mod\calC/G$, or equivalently to show the commutativity
of the right square of the diagram
$$
\begin{CD}
\Ds_{\al\in G} M(\al x)@=(P\down M)(P x) @>{F_{P x}}>> N(P x)\\
@V{(M(f_{\al,\be}))_{(\al,\be)}}VV @V{(P\down M)(f)}VV @VV{N(f)}V\\
\Ds_{\al\in G} M(\al y)@=(P\down M)(P y) @>>{F_{P y}}> N(P y)
\end{CD}
$$
for all $f\in \calC/G(P y, P x)$ and $x, y \in \calC$.
Now there exists a unique
$(f_{\be})_{\be\in G}\in \Ds_{\be\in G}\calC(\be y, x)$
such that
$f= P^{(1)}_{y,x}((f_{\be})_{\be\in G})
= \sum_{\be\in G}P^{(1)}_{y,x}(f_{\be})$.
Since it is enough to verify this commutativity for each
term $P^{(1)}_{y,x}(f_{\be})$ of $f$, we may assume that
$\deg f = \be$ for some $\be\in G$.
Then to show this commutativity
it suffices to show that the right square of the following
diagram commutes for each $\al \in G$:
$$
\begin{CD}
N^1(P\al x)@=(P\down M)^{\al}(P x) @>{F_{\al,P x}}>> N^{\al}(P x)\\
@V{N^1(P(f_{\al,\al\be}))}VV @V{M(f_{\al,\al\be})}VV @VV{N(f)}V\\
N^1(P\al\be y)@=(P\down M)^{\al\be}(P y) @>>{F_{\al\be,P y}}> N^{\al\be}(P y)
\end{CD}
$$
or equivalently that
$$N(\ph_{\al\be,y})N(P(f_{\al,\al\be})) = N(f)N(\ph_{\al,x}).$$
This holds if the equation
$$
P(f_{\al,\al\be})\ph_{\al\be,y} = \ph_{\al,x}f
$$
holds.
Now since $\deg f = \be$, $f$ has the form
$f=(\de_{\mu\inv\la,\be}f_{\mu,\la})_{(\la,\mu)}$.
Using this a direct calculation shows that both hand sides
of this equation are equal to
$(\de_{\la,\nu\al\be}f_{\nu\al,\la})_{\la,\nu}$.
\end{proof}

\begin{rmk}
Ignoring the relationship with pushdown functors,
we have an alternative proof of the theorem above
by Theorem \ref{liberalization} and \cite[Theorem 4.5]{C-M} as follows:
$\Mod \calC \simeq \Mod (\calC/G)\# G \simeq \Mod_{G}(\calC/G)$,
and in fact, it is possible to prove the theorem by showing that the composite
of these equivalences is equal to $P'\down$ (using an equivalence
$\Mod \calC \to \Mod (\calC/G)\# G$ induced by the equivalence
$\ep'_{\calC}\colon (\calC/G)\# G \to \calC$ that is a quasi-inverse of
$\ep_{\calC}$ and is defined in \cite[Definition 7.3]{Asa09}).
We kept the proof above because
it gives the explicit form of a quasi-inverse of $P'\down$
(see constructions of $f$ and $M$ in the proof)
.
\end{rmk}

\begin{cor}
Let $R$ be an algebra with a $G$-action.
Then we have an equivalence $\Mod R \simeq \Mod_{G}R*G$.\hfil \qed
\end{cor}

\begin{cor}
The pushdown functor $P\down$ induces a $G$-covering functor
$$
\mod \calC \to \mod_{(G)}(\calC/G)
,$$
where $\mod_{(G)}(\calC/G)$ is the full subcategory of $\mod (\calC/G)$
consisting of $\Fgt(M,d)$ with $(M,d) \in \mod_G(\calC/G)$, namely
of $G$-gradable $\calC/G$-modules.
Therefore, in particular, we have
$$
(\mod \calC)/G \simeq \mod_{(G)}(\calC/G)
\quad\text{ and}\quad
\mod \calC \simeq (\mod_{(G)} (\calC/G))\# G%
.
$$
\end{cor}

\begin{proof}
This follows by Theorems \ref{chrtrz-cov}, 
\ref{pushdown-precov}, and \ref{remov-grading},
and by Corollary \ref{covering-orbit-smash}.
\end{proof}

\begin{rmk}
In the above a similar statement holds for $\Kb(\Prj\calC)$.
\end{rmk}

\section{Colimit orbit categories}
\label{section:colim-orb-cat}

In this section we investigate the orbit category of a category $\C$
by a cyclic group $G$ generated by an auto-equivalence of $\C$
modulo natural isomorphisms.
Throughout this section let $S\colon \calC \to \calC$
be an auto-equivalence of $\calC$.
The point to define the orbit category $\calC/\ang{\bar{S}}$
is in replacing $S$ by an {\em automorphism} $S'$
of some category $\calC'$
with an equivalence $H\colon \calC \to \calC'$
having the property that the diagram
$$
\begin{CD}
\calC @>{S}>> \calC\\
@V{H}VV @VV{H}V\\
\calC' @>>{S'}> \calC'
\end{CD}
$$
commutes up to natural isomorphisms, and then we can define
$\calC/\ang{\bar{S}}$ by setting $\calC/\ang{\bar{S}}:= \calC'/\ang{S'}$.
In Remark \ref{rmk:autoequivalence} the category $\calC'$ was taken as
a basic subcategory of $\calC$ with $H$ a quasi-inverse of
the inclusion functor $\calC' \to \calC$.
There is an alternative choice for $\calC'$ as used in the paper \cite{K-V}
by Keller and Vossieck.
We realized that their choice of $\calC'$ has the form $\calC_{/S} \# \bbZ$
for some $\bbZ$-graded category $\calC_{/S}$, which we call
the {\em colimit orbit category} of $\calC$ by $S$.
As a consequence, we have $\calC/\ang{\bar{S}}:=\calC'/\ang{S'}
\simeq (\calC_{/S}\#\bbZ)/\bbZ \simeq \calC_{/S}$. Thus
the orbit category $\calC/\ang{\bar{S}}$ is justified as the
colimit orbit category.

\begin{dfn}
(1) We define a $\bbZ$-graded category $\calC_{/S}$ called
the {\em colimit orbit category} of $\calC$ by $S$ as follows.
\begin{itemize}
\item
$\obj(\calC_{/S}):= \obj(\calC)$;
\item
For each $X, Y \in \obj(\calC_{/S})$,
$\calC_{/S}(X, Y):= \Ds_{r\in \bbZ}\calC_{/S}^r(X, Y)$, where
$$\calC_{/S}^r(X, Y):=\colim_{m\ge r}\calC(S^{m-r}X, S^{m}Y);$$
\item
For each composable morphisms $X\ya{f} Y \ya{g} Z$ in $\calC_{/S}$,
say $f=(f_a)_{a\in \bbZ}$ and $g=(g_b)_{b\in \bbZ}$,
$$gf:= (\sum_{c=a+b}g_{b}f_{a})_{c\in \bbZ}.$$
\end{itemize}
(2) We define a functor
$S'\colon \calC_{/S}\#\bbZ \to \calC_{/S}\#\bbZ$ as follows.
\begin{itemize}
\item
For each $X^{(i)}\in \obj(\calC_{/S}\#\bbZ)$,
$S'X^{(i)}:= X^{(i-1)}$;
\item
For each $X^{(i)}, Y^{(j)} \in \obj(\calC_{/S}\#\bbZ)$, define
$$
S'\colon (\calC_{/S}\#\bbZ)(X^{(i)}, Y^{(j)}) \to
(\calC_{/S}\#\bbZ)(X^{(i-1)}, Y^{(j-1)})
$$
as the identity map of
$\calC_{/S}^{i-j}(X,Y) =(\calC_{/S}\#\bbZ)(X^{(i)}, Y^{(j)})
=(\calC_{/S}\#\bbZ)(X^{(i-1)}, Y^{(j-1)})$.
\end{itemize}
(3) We define a functor $H \colon \calC \to \calC_{/S}\# \bbZ$ as follows.
\begin{itemize}
\item
For each $X \in \calC$, $HX:=X^{(0)}$.
\item
For each $X, Y \in \calC$, define
$$H\colon \calC(X, Y) \to (\calC_{/S}\# \bbZ)(X^{(0)}, Y^{(0)})
=\calC_{/S}^0(X,Y)=\colim_{m\ge 0}\calC(S^m X, S^m Y)$$
by $Hf:=[f]$, the image of $f$ in $\colim_{m\ge 0}\calC(S^m X, S^m Y)$
for each $f \in \calC(X, Y)$.
\end{itemize}
\end{dfn}

\setcounter{clm}{0}
\begin{prp}
$(1)$ $S'$ is an automorphism of the category $\calC_{/S} \# \bbZ$;

$(2)$ $H$ is an equivalence; and

$(3)$ We have a commutative diagram
$$
\begin{CD}
\calC @>{S}>> \calC\\
@V{H}VV @VV{H}V\\
\calC_{/S}\# \bbZ @>>{S'}> \calC_{/S}\# \bbZ.
\end{CD}
$$
up to natural isomorphisms.
\end{prp}

\begin{proof}
(1) Define a functor
$F\colon \calC_{/S}\#\bbZ \to \calC_{/S}\#\bbZ$ as follows.
\begin{itemize}
\item
For each $X^{(i)}\in \obj(\calC_{/S}\#\bbZ)$,
$FX^{(i)}:= X^{(i+1)}$;
\item
For each $X^{(i)}, Y^{(j)} \in \obj(\calC_{/S}\#\bbZ)$, define
$$
F\colon (\calC_{/S}\#\bbZ)(X^{(i)}, Y^{(j)}) \to
(\calC_{/S}\#\bbZ)(X^{(i+1)}, Y^{(j+1)})
$$
as the identity map of
$\calC_{/S}^{i-j}(X,Y) =(\calC_{/S}\#\bbZ)(X^{(i)}, Y^{(j)})
=(\calC_{/S}\#\bbZ)(X^{(i+1)}, Y^{(j+1)})$.
\end{itemize}
Then it is obvious that $F$ is the inverse of $S'$, and hence
$S'$ is an automorphism of $\calC_{/S}\# \bbZ$.

(2) It is obvious that $H$ is fully faithful because so is $S$.
\begin{clm}
\label{nonpositive}
For each $X^{(-i)} \in \calC_{/S}\# \bbZ$ with $i \ge 0$,
we have $X^{(-i)} \iso (S^{i} X)^{(0)}$.
\end{clm}

Indeed, $[\id_{X}]\in \colim_{m\ge -i}\calC(S^{m+i}X, S^{m+i} X)
= (\calC_{/S}\# \bbZ)(X^{(-i)}, (S^{i} X)^{(0)})$
is an isomorphism in $\calC_{/S}\# \bbZ$.

\begin{clm}
\label{positive}
For each $X \in \calC_{/S}$ and each $i \in \bbZ$ with $i >0$,
we have $(S^i X)^{(i)} \iso X^{(0)}$.
\end{clm}

Indeed, $[\id_{S^i X}] \in \colim_{m\ge i}\calC(S^m X, S^m X) =
(\calC_{/S}\# \bbZ)((S^i X)^{(i)}, X^{(0)})$
is an isomorphism in $\calC_{/S}\# \bbZ$.

Using these we show that $H$ is dense.
Let $X^{(i)} \in \calC_{/S}\#\bbZ$.
If $i \le 0$, then $X^{(i)} \iso H(S^{-i}X)$ by Claim \ref{nonpositive}.
If $i > 0$, then there is some $Y \in \calC$ such that
$X \iso S^i Y$ in $\calC$ because $S$ is dense; and
we have $X^{(i)} \iso (S^i Y)^{(i)} \iso H(Y)$ by Claim \ref{positive}.
Hence $H$ is dense, and is an equivalence.

(3) By Claim \ref{nonpositive}, we have an isomorphism
$$[\id_{X}] \colon S'H(X) = X^{(-1)} \to (S X)^{(0)} = HS(X).$$
Then it is easy to see that $([\id_{X}])_{X\in \calC}\colon S'H \to HS$
is a natural isomorphism.
\end{proof}

By this statement we can define the ``orbit category''
$\calC/\ang{\bar{S}}$
as follows.

\begin{dfn}
$\calC/\ang{\bar{S}}:= (\calC_{/S}\# \bbZ)/\ang{S'}$.
\end{dfn}

Since the $\bbZ$-action on $\calC_{/S}\# \bbZ$ defined by $n \mapsto {S'}^{-n}$ ($n \in \bbZ$)
coincides with the canonical $\bbZ$-action on it,
we obtain the following by Proposition \ref{smash-orb}.

\begin{thm}
$\calC/\ang{\bar{S}} \simeq \calC_{/S}$.
\end{thm}
\section{Quiver presentations of skew monoid categories}
\def\t{\mathbf{t}}
\def\h{\mathbf{h}}

In this section we compute a quiver presentation of the first orbit category $A\fstorbit G$
of a category $A$ and a monoid $G$,
where $A$ is given by a quiver with relations over a field
and $G$ is given by a monoid presentation.
We refer the reader to Howie's book \cite{How} for monoid presentations.
To be precise,
we assume the following setting throughout this section:
\begin{enumerate}
\item $\k$ is a field;
\item $Q:=(Q_{0},Q_{1},\t,\h)$ is a locally finite quiver;
\item $\k[Q]$ is the {\em path category} of $Q$ over $\k$;
\item $\rho$ is a set of morphisms of $\k[Q]$
such that $\ang{\ro} \cap \{e_x \mid x \in Q_0\} = \emptyset$,
where $\ang{\ro}$ is the ideal of $\k[Q]$ generated by $\ro$;
\item $A:=\k[Q,\ro]:= \k[Q]/\ang{\ro}$;
\item $G$ is a monoid with a monoid presentation $G=\ang{S\mid R}$
(even when $G$ is a group we use a monoid presentation);
\item When $G$ has a zero, we add a zero object $0$ to $A$ to form
the category $A \cup \{0\}$, which we denote also by $A$
(Note that $\ds(A \cup \{0\})$ and $\ds A$ are isomorphic as algebras);  and
\item $G$ acts on $A$ by 
a homomorphism 
$G \to \End(A)$.
\end{enumerate}

In (3) recall the definition of the path category $\k[Q]$.
\begin{itemize}
\item
$\obj(\k[Q]) := Q_0$;
\item
For each $x, y\in \k[Q]$,
$\k[Q](x,y)$ is the $\k$-vector space with basis the set of paths from $x$ to $y$
in $Q$; and
\item
The composition of morphisms is given by the composition of paths
as in the definition of the multiplication of the path algebra $\k Q$.
\end{itemize}

Thus we have
$\k(Q,\ro)\iso \ds A$ (see Sect.~3 for the definition); and
$A(x,y)= e_y \k(Q,\ro)e_x$ for all $x,y\in Q_0$,
where $e_x$ is the path of length 0 at each vertex $x \in Q_0$.
The algebra $\k(Q,\ro)$ and the category $A$ are presented
by the same quiver with relations, and we often identify them.

\begin{thm}\label{quiver-presentation}
In the above setting, the category $A\fstorbit G$ and the algebra
$\ds (A\fstorbit G)$ are presented by the following
quiver $Q'$ and the following three kinds of relations:

{\bf Quiver:} $Q'$ is the quiver obtained from $Q$ by adding new arrows
$$(S \times Q_0)':=\{(g,x) \colon x \to gx \mid g \in S, x \in Q_0, gx \ne 0\}.$$
Namely, the quiver $Q'=(Q'_0, Q'_1, \t', \h')$ is defined as follows.
$$
\begin{aligned}
Q_0'&:= Q_0,\\
Q_1'&:= Q_1 \sqcup (S \times Q_0)',\\
(\t'(\al), \h'(\al))&:=(\t(\al),\h(\al)),\quad \forall \al \in Q_1, \\
(\t'(g,x),\h'(g,x))&:=(x,gx),\quad \forall (g,x)\in (S \times Q_0)', 
\end{aligned}
$$
where $\sqcup$ denotes the disjoint union.

{\bf Relations:}

$(1)$ The relations in the category $A$: $\mu = 0$, $\forall \mu \in \ro$;

$(2)$ Skew monoid relations: $(g,y)\al=g(\al)(g,x)$,
$\forall \al\colon x \to y$ in $Q_1$, $\forall g\in S$; and

$(3)$ The relations in the monoid $G$: $\pi(g,x)=\pi(h,x)$,
$\forall (g,h) \in R$, $\forall x \in Q_0$,

where for each $x \in Q_0$ and for each $g \in G\setminus \{0,1\}$, say
$g=g_t\cdots g_1$ $(g_1,\dots, g_t \in S$, $t \ge 1)$ we set 
$\pi(g,x)$ to be the path
$\pi(g,x):=(g_t, g_{t-1}\dots g_1x)\cdots (g_2,g_1x)(g_1,x)$ in $Q'$. Namely, it has the form
$$
gx \xleftarrow{(g_t, g_{t-1}\dots g_1x)} \cdots
\xleftarrow{(g_3,g_2g_1x)} g_2g_1x \xleftarrow{(g_2,g_1x)} g_1 x \xleftarrow{(g_1,x)} x,
$$
and we set $\pi(1,x):= e_x$, $\pi(0,x):=0$.
\end{thm}

\begin{proof}
It is enough to prove the assertion for the algebra $\ds (A\fstorbit G)$.
Note that $S^*$ acts on $A$ by
$S^* \ya{\text{can}} G \to \End(A)$.
For each $\mu\in \k Q$,
we set $\tilde{\mu}:=\mu+\langle\rho\rangle\in A$,
and $\tilde{Q}_{0}:=\{\tilde{e}_{x}\mid x\in Q_{0}\}$.
For each $g \in S^*$ we set $\bar{g}:=R^{\#}g \in G$.
We may assume that $\bar{g} \ne \bar{h}$
if $g \ne h$ for all $g, h \in S$.
Define an ideal $I$ of $\k Q'$ by
\begin{equation*}
\begin{split}
I:= \ang{\ro}_{\k Q'}+\ang{(g,y)\al - g(\al)(g,x) \mid \al\colon x \to y
\text{ in }Q_1, g\in S}_{\k Q'}\\
+ \ang{\pi(g,x) - \pi(h,x)\mid (g,h) \in R, x \in Q_0}_{\k Q'},
\end{split}
\end{equation*}
where in the second term we choose $g(\al)$ as an element of $g(\tilde{\al})$
for each $\al \in Q_1$.
Note that $I$ is determined independent of the choice because
$\ang{\ro} \subseteq I$.

Let $x, y \in Q_0$.
For each $\bar{g} \in G$ let $\si_{\bar{g}}$ be the canonical inclusion
$A(\bar{g}x, y) \to \Ds_{\bar{h} \in G} A(\bar{h}x, y) = (A \fstorbit G)(x, y)$ defined by
$\si_{\bar{g}}(f):= (f\de_{\bar{g}, \bar{h}})_{\bar{h}\in G}$
for all $f \in A(\bar{g}x, y)$.
We set $f * \bar{g} := \si_{\bar{g}}(f)$.

First define a $\k$-algebra homomorphism
$\Ps \colon \k Q' \to \ds(A\fstorbit G)$ by
$$
\begin{aligned}
e_x &\mapsto \tilde{e}_x \ (:= \tilde{e}_x*1_G),
\quad \forall x \in Q_0,\\
\al &\mapsto \tilde{\al} \ (:= \tilde{\al}*1_G),
\quad \forall \al \in Q_1,\\
(g,x) &\mapsto \tilde{e}_{gx}*\bar{g},
\quad \forall (g,x) \in (S \times Q_0)'.
\end{aligned}
$$
Then since
$\k Q'$ is isomorphic to the quotient of the free associative algebra
$\k\ang{Q'_0 \sqcup Q'_1}$
modulo the ideal generated by
the set
$$
\{e_xe_y - \de_{x,y}e_x, e_y \al e_x - \al, e_{gx}(g,x) e_x - (g,x)
\mid x,y\in Q_0, \al \colon x \to y \text{ in } Q_1, (g,x)\in (S\times Q_0)'\}
$$
and since in $\ds(A\fstorbit G)$ we have relations
$$
\tilde{e}_x \tilde{e}_y = \de_{x,y}\tilde{e}_x,
\tilde{e}_y\tilde{\al}\tilde{e}_x = \tilde{\al},
\tilde{e}_{gx}(\tilde{e}_{gx}*\bar{g})\tilde{e}_x =\tilde{e}_{gx}*\bar{g}
$$
for all $x,y\in Q_0$, $\al\colon x \to y$ in $Q_1$,
and $(g,x) \in (S \times Q_0)'$,
we see that $\Ps$ is well-defined.

\setcounter{clm}{0}
\begin{clm}
$\Ps(I)=0$.
\end{clm}
Indeed, first $\Ps(\ro)=0$ shows $\Ps((\k Q')\ro(\k Q'))=0$.
Second, for each $\al \colon x \to y$ in $Q_1$ and $g\in S$,
we have $\Ps((g,y)\al - g(\al)(g,x))
=(\tilde{e}_{gy}*\bar{g})\tilde{\al}
- g(\tilde{\al})(\tilde{e}_{gx}*\bar{g})
=\tilde{e}_{gy} g(\tilde{\al})*\bar{g} - g(\tilde{\al})*\bar{g} = 0$.
Finally, for each $(g,h) \in R$ and $x \in Q_0$
we have
$$
\begin{aligned}
\Ps(\pi(g,x))
&=
(\tilde{e}_{gx}*\bar{g}_t)\cdots (\tilde{e}_{g_2g_1x}*\bar{g}_2)
(\tilde{e}_{g_1x}*\bar{g}_1)\\
&=
(\tilde{e}_{gx}*\bar{g}_t)\cdots (\tilde{e}_{g_2g_1x}
\tilde{e}_{g_2 g_1x}*\bar{g}_2\bar{g}_1)\\
&\ \ \vdots\\
&=\tilde{e}_{gx}*\bar{g}
\end{aligned}
$$
if $g=g_t\cdots g_1$ ($t\ge 1$).
Also $\Ps(\pi(g,x))=\tilde{e}_x=\tilde{e}_{gx}*\bar{g}$
if $g=1$.
Thus in any case we have
\begin{equation}\label{im-pi}
\Ps(\pi(g,x))=\tilde{e}_{gx}*\bar{g}.
\end{equation}
Similarly, $\Ps(\pi(h,x))=\tilde{e}_{hx}\bar{h}$.
Since $(g,h) \in R$, we have $\bar{g}=\bar{h}$, and
$\tilde{e}_{gx}\bar{g}=\tilde{e}_{hx}\bar{h}$.
Hence $\Ps(\pi(g,x)-\pi(h,x))=0$.
As a consequence, we have $\Ps(I)=0$.

By Claim 1 the homomorphism $\Ps$ induces a $\k$-algebra homomorphism
$\Ph\colon \k Q'/I \to \ds(A\fstorbit G)$.
It is enough to show that $\Ph$ is an isomorphism.

Next we fix a $\k$-basis of $\ds(A\fstorbit G)$.
Since $A = \sum_{\mu\in \bbP Q} \k\tilde{\mu}$,
there exists a $\k$-basis $M$ of $A$
that is contained in $\bbP Q$.
Thus $\{\tilde{\mu}*\bar{g} \mid \mu \in M, \bar{g} \in G\}\setminus\{0\}$
forms a $\k$-basis of $\ds(A\fstorbit G)$.

\begin{clm}
$\calM:=\{
\tilde{\mu}*\bar{g} \mid \mu \in M, \bar{g} \in G, \t(\mu)\in g(Q_0)
\}$
forms a $\k$-basis of $\ds(A\fstorbit G)$.
\end{clm}
Indeed,
for each $\mu \in M, \bar{g} \in G$ and for each $x,y\in Q_0$ we have
$$
\begin{aligned}
(\tilde{e}_y*1_G)(\tilde{\mu}*\bar{g})(\tilde{e}_x*1_G)
&=(\tilde{e}_y*1_G)(\tilde{\mu}\bar{g}(\tilde{e}_x)*\bar{g})\\
&=\tilde{e}_y\tilde{\mu}\bar{g}(\tilde{e}_x)*\bar{g}\\
&=\begin{cases}
\tilde{e}_y\tilde{\mu}\tilde{e}_{gx}*\bar{g} &\text{if $gx\ne 0$}\\
0&\text{if $gx=0$}
\end{cases}
\end{aligned}
$$
Therefore
$(\tilde{e}_y*1_G)(\tilde{\mu}*\bar{g})(\tilde{e}_x*1_G)\ne 0$
if and only if $\t(\mu)=gx\in g(Q_0)$ and $\h(\mu)=y$; and
in this case, we have $(\tilde{e}_y*1_G)(\tilde{\mu}*\bar{g})(\tilde{e}_x*1_G)=\tilde{\mu}*\bar{g}$.
This proves the claim.

\begin{clm}
For each $g,h\in S^*$ and $x \in Q_0$,
if $\bar{g}= \bar{h}$ in $G$, then
$\ovl{\pi(g,x)}=\ovl{\pi(h,x)}$ in $\k Q'/I$.
\end{clm}

Indeed, the fact that $\bar{g}=\bar{h}$ in $G$
is equivalent to saying that
$(g,h) \in R^{\#}$.
If $g=h$ in $S^*$, then the assertion is obvious.
Otherwise, there is a sequence of elementary $R$-transitions
connecting $g$ and $h$.
Therefore we may assume that
there exist $(a,b) \in R$ and $c,d\in S^*$ such that
$g=cad, h=cbd$.
Note that we have $adx:=\bar{a}\bar{d}x=\bar{b}\bar{d}x=:bdx$ because $\bar{a}=\bar{b}$. 
Then
$$
\begin{aligned}
\pi(g,x) - \pi(h,x)
&=
\pi(cad,x) - \pi(cbd,x) \\
&=
\pi(c,adx)\pi(a,dx)\pi(d,x)  - \pi(c,bdx)\pi(b,dx)\pi(d,x)\\
&=
\pi(c,adx)(\pi(a,dx) - \pi(b,dx))\pi(d,x) \in I.
\end{aligned}
$$
This proves the claim.

For each $\bar{g} \in G$ with $g\in S^*$, we define
\begin{equation}
\ovl{\pi(\bar{g},x)}:= \ovl{\pi(g,x)},\end{equation}
which is well-defined by Claim 3.

\begin{clm}
Let $x,x',y\in Q_0$ and $\bar{g} \in G$ with $g\in S^*$.
If $gx = y = gx'$,
then $x=x'$.
Hence for each $y \in g(Q_0)$ the inverse image
of $y$ under $\bar{g}$ has exactly one element, which we denote
by $\bar{g}\inv(y)$.
\end{clm}

Indeed, $gx = y = gx'$ shows
$\tilde{e}_{gx} = \tilde{e}_y = \tilde{e}_{gx'} \in A$.
Assume that $x \ne x'$.
Then $\tilde{e}_y = \tilde{e}_y \tilde{e}_y
=\tilde{e}_{gx}\tilde{e}_{gx'}
= g(\tilde{e}_x \tilde{e}_{x'})=g(0)=0$.
But since $A=\k Q/\ang{\ro}$ and 
$\ang{\ro} \cap \{e_x \mid x \in Q_0\} = \emptyset$,
we have $\tilde{e}_y \ne 0$, a contradiction.
Hence we must have $x=x'$.

\begin{clm}
Let $\et \in \bbP Q'$. Then
$\ovl{\et}$ is a linear combination of elements of $\k Q'/I$ of
the form $\ovl{\la\pi(\bar{g},\bar{g}\inv(\t\la))}$
for some $\bar{g}\in G$ and $\la \in M$ with $\t\la \in g(Q_0)$.
$($Note that the element $\bar{g}\inv(\t\la) \in Q_0$ is well-defined by
Claim $4.)$
\end{clm}

Indeed, for each arrow $\al\colon x \to y$ in $Q_1$ we have
\begin{equation}\label{to-right}
\ovl{(g,y)\al} = \ovl{g(\al)(g,x)}\quad\text{in $\k Q'/I$}
\end{equation}
by definition of $I$.
In the path $\et$ by using (\ref{to-right})
we can move factors
of the form $\ovl{(g,y)}$ (with $(g,y)\in (S\times Q_0)'$)
to the right, and finally we have
\begin{equation}\label{norm-form-et}
\ovl{\et}= \ovl{e_y\al_s\cdots\al_1(g_t,x_t)\cdots (g_1,x_1)}
\end{equation}
for some $\al_i \in Q_1$, $g_i \in S$,
$x_i, y \in Q_0$,
where the paths in the right hand side is composable.
Set $g:=g_t\cdots g_1 \in S^*$ and $\la:=e_y\al_s\cdots \al_1$.
Then the composability of the right hand side of (\ref{norm-form-et})
implies that
$$
\ovl{\pi(g,x_1)}=\ovl{(g_t,x_t)\cdots (g_1,x_1)}, \la\in \bbP Q,
\text{ and $\t(\la)=gx_1 \in g(Q_0)$}.
$$
Here $\t(\la)=gx_1$ implies that $x_1=\bar{g}\inv\t(\la)$
by Claim 4.
Hence
$$
\ovl{\et}=\ovl{e_y\al_s\cdots\al_1(g_t,x_t)\cdots (g_1,x_1)}
=\ovl{\la\pi(\bar{g},\bar{g}\inv\t(\la))}.
$$
Now since $M$ is a $\k$-basis of $A$, $\la$ is expressed as a
linear combination of paths in $M$ with the same tail as $\la$
and with the same head as $\la$.
By replacing $\la$ by this linear combination, we obtain the
required expression of $\ovl{\et}$.

\begin{clm}
The set $\calS:=\{\ovl{\mu\pi(\bar{g},\bar{g}\inv\t(\mu))} \mid
\mu \in M, \bar{g}\in G, \t(\mu) \in g(Q_0)\}$
spans $\k Q'/I$.
\end{clm}
Indeed, this is clear from Claim 5.

For each $\mu \in M$, each $\bar{g} \in G$, we have
$$
\Ph(\ovl{\mu\pi(\bar{g},\bar{g}\inv\t(\mu))})=
\tilde{\mu}\tilde{e}_{\t(\mu)}*\bar{g} = \tilde{\mu}*\bar{g}
$$
by \eqref{im-pi}.
Hence the restriction $\Ph\mid_{\calS}\colon \calS \to \calM$
is surjective, and hence so is $\Ph\colon \k Q'/I \to \ds(A\fstorbit G)$.
\begin{clm}
$\calS$ is a $\k$-basis of $\k Q'/I$.
\end{clm}
Indeed, it is enough to show that $\calS$ is linearly independent.
Assume
$$
\sum_{\bar{g}\in G, \mu\in M, \t(\mu) \in g(Q_0)}
t_{\bar{g},\mu}\ovl{\mu\pi(\bar{g},\bar{g}\inv\t(\mu))}=0
$$
in $\k Q'/I$ with $t_{\bar{g},\mu}\in \k$.
Then by applying $\Ph$ to this equality we have
$$
\sum_{\bar{g}\in G, \mu\in M, \t(\mu) \in g(Q_0)}
t_{\bar{g},\mu}\tilde{\mu}*\bar{g}=0
$$
By Claim 2, we have all coefficients $t_{\bar{g},\mu}$ are zero.

By Claim 7 we see that $\Ph \colon \k Q'/I \to \ds(A\fstorbit G)$
is a bijection, i.e., an isomorphism.
\end{proof}

\begin{rmk}\label{computation-orbitcat}
(1)
By Definition \ref{dfn:skew-monoid-cat} we can compute a quiver presentation
of the skew monoid algebra $\k(Q,\ro) * G$ using
the computation of $A\fstorbit G$ described in the theorem above.
When $Q_0$ is finite and $G$ is a group, this skew monoid algebra coincides with
the usual skew group algebra $\k(Q,\ro) * G$, and hence the theorem above
gives also a way to compute skew group algebras.

(2)
Let $G$ be a group acting freely on a basic category $\calC$ as in
the classical covering setting.
Then we can form the orbit category $\calC\oorbit G$ (see Remark \ref{oldorbit} (1)),
which is expressed as
$\calC\oorbit G \simeq \bas(\calC\fstorbit G)$.
Hence we can use
the theorem above to compute also  $\calC\oorbit G$.

(3)
For a quiver $Q$, recall that an ideal $J$ of $\k[Q]$ (resp.\ of $\k Q$)
is called {\em admissible}
if $(\k[Q]^{+})^h \le J \le (\k[Q]^{+})^2$ for some $h \ge 2$, where
$\k[Q]^{+}$ is the ideal generated by all arrows in $Q$.
If $J$ is admissible, then $A:=k[Q]/J$ is basic, the radical of $A$ is given by
$k[Q]^+/J$, the ordinary quiver of $A$ is again $Q$, and all idempotents
given by vertices are primitive.
Note that in the theorem above the ideal $I$ of $\k[Q']$ (resp.\ of $\k Q'$)
generated by the relations in our presentation is not always admissible
even if we start from the setting that the ideal $\ang{\ro}$ is admissible.
This may seem to be a weak point because
sometimes the idempotents given by vertices are not primitive,
and the quiver is not determined uniquely,
but this enables us to handle also non-basic categories, and 
anyway at lease it gives a quiver presentation, so if necessary, one can later adjust the
presentation to admissible one as seen in the next section.
%
\end{rmk}

\section{Examples}
Throughout this section $\k$ is a field.
\subsection{Classical example}
We begin with the following classical example in \cite{RR}.
\begin{exm}
Example of $(\calC/G)\# G \simeq \calC$.
Let $G:=\ang{g\mid g^2=1}$ be the cyclic group of order 2,
$Q$ the following quiver:
$$
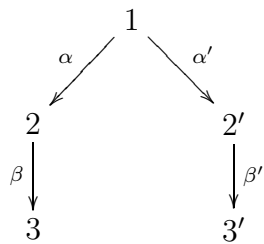
\xymatrix{
& 1 \ar[dl]_{\al}\ar[dr]^{\al'}\\
2\ar[d]_{\be} && 2'\ar[d]^{\be'}\\
3 && 3'
}
$$
Define an action of $g$ on $Q$ by the permutation
$\pmat{1\ 2\ 2'\ 3\ 3'\\1\ 2'\ 2\ 3'\ 3}$ $= (2\ 2')(3\ 3')$
of vertices of $Q$, and define an action of $g$ on the
category $\calC:= \k Q$ by the linearization of this.
Clearly this action is not free.
We compute $\calC/G$ and $\calC *G$ by using
Theorem \ref{quiver-presentation}.
First $\calC/G$ is given by the following quiver
$$
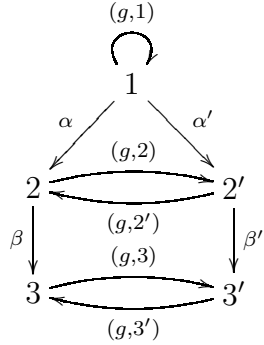
\xymatrix{
& 1 \ar@(ul,ur)^{(g,1)}\ar[dl]_{\al}\ar[dr]^{\al'}\\
2\ar[d]_{\be} \ar@/^/[rr]^{(g,2)}&& 2'\ar[d]^{\be'}\ar@/^/[ll]^{(g,2')}\\
3 \ar@/^/[rr]^{(g,3)}&& 3'\ar@/^/[ll]^{(g,3')}
}
$$
with the following relations:

{\bf From $g^2=1$:}
$
(g,1)^2 = e_1,
\left\{\begin{aligned}(g,2')(g,2)&=e_2\\(g,2)(g,2')&=e_{2'}\end{aligned}\right.,
\left\{\begin{aligned}(g,3')(g,3)&=e_3\\(g,3)(g,3')&=e_{3'}\end{aligned}\right..
$

{\bf From skew group relations:}
$\left\{\begin{aligned}&(g,2)\al=\al'(g,1)\\&(g,2')\al'=\al(g,1)\end{aligned}\right.,
\left\{\begin{aligned}&(g,3)\be=\be'(g,2)\\&(g,3')\be'=\be(g,2')\end{aligned}\right..
$
Then $(\calC/G)\# G$ is just the ``double'' $M_2(\calC)$ of $\calC$,
i.e., the category obtained from $\calC$ by adding one object isomorphic to $x$
for each object $x$ in $\calC$.
By Theorem \ref{liberalization}
this $(\calC/G)\# G$ is $G$-equivariantly equivalent to the original algebra $\calC$, and is a ``liberalization'' of the $G$-action of $\calC$.

Here the algebra $\calB:= \bas(\calC/G)$ is given by the following quiver with relations:
$$
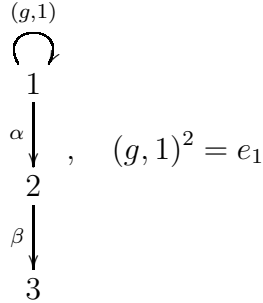
\vcenter{
\xymatrix{
1\ar@(ul,ur)^{(g,1)}\ar[d]_{\al}\\
2\ar[d]_{\be}\\
3
}},\quad (g,1)^2 = e_1
$$
When $\chr \k \ne 2$, we have $\bas(\calC/G)(1,1)=\k\ep_1\times \k\ep_2$, where
$\ep_1 :=\frac{1}{2}(e_1 + (g,1)), \ep_2:= \frac{1}{2}(e_1 - (g,1))$
by Chinese Remainder Theorem.
Hence $\calC * G$ is given by the following quiver
with no relations:
$$
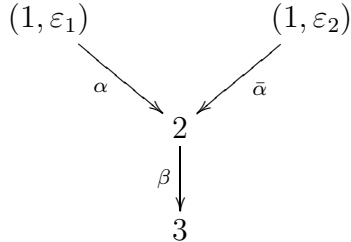
\xymatrix{
(1,\ep_1) \ar[rd]_{\al} &&(1,\ep_2)\ar[dl]^{\bar{\al}}\\
&2\ar[d]_{\be}\\
&3
}
$$
As well known \cite{RR} if we define an action of $g$
to this algebra by exchanging $(1,\ep_1)$ and $(1,\ep_2)$, then
the skew group algebra $(\calC *G)*G$ is isomorphic to the original
algebra $\calC$, which can be checked by the same way as above.
\end{exm}

\begin{exm}
\label{exm:liberalization}
Example of $(\calB \# G)/G \simeq \calB$.
Consider the algebra $\calB$ in the previous example as a category
with no assumption on $\chr \k$.
Note that the $G$-grading of $\calB$ is given by
$\deg(\al) = \deg(\be) = g^0=1$ and $\deg(x)=g^{-1} = g$, where we
put $x:=(g,1)$.
Then $\calB\# G$ is given by the following quiver with relations: 
$$
\vcenter{
\xymatrix{
1^{(1)}\ar@<1ex>[r]^{x^{(1)}}\ar[d]_{\al^{(1)}} & 1^{(g)}\ar@<1ex>[l]^{x^{(g)}}\ar[d]^{\al^{(g)}}\\
2^{(1)}\ar[d]_{\be^{(1)}} &2^{(g)}\ar[d]^{\be^{(g)}}\\
3^{(1)} & 3^{(g)}
}},\quad x^{(g)}x^{(1)} = e_{1^{(1)}}, x^{(1)}x^{(g)} = e_{1^{(g)}}
$$
whose $G$-action is given by the permutation
$(1^{(1)}\ 1^{(g)})(2^{(1)}\ 2^{(g)})(3^{(1)}\ 3^{(g)})$ and is free.
This is $G$-equivariantly equivalent to the original algebra $\calC$,
and is a smaller ``liberalization'' of the $G$-action of $\calC$.
$(\calB \# G)/G$ is also the double $M_2(\calB)$ of $\calB$
and is equivalent to $\calB$.
If again $\chr \k \ne 2$,
we see that $(\calC/G)/G \simeq (\calC/G)\# G \ (\simeq \calC)$
this explains the phenomenon above that $(\calC *G)*G \iso \calC$.
\end{exm}

\subsection{Infinite cyclic group}
\begin{exm}
Let $p:=\chr \k$ and
$A:=\k[\al]/(\al^3)$, namely the algebra given by the following
quiver with relations:
$$
\xymatrix{1 \ar@(ul,dl)_{\al}}, \quad \al^3 = 0.
$$
Further let $g$ be the automorphism of $A$ defined by $g(1):=1$ and
$g(\al):= \al + \al^2$, and set $G$ to be the cyclic group
generated by $g$.
Then $G$ has the presentation
$$
G = 
\begin{cases}
\ang{g\mid g^p=1} &\text{if } p>0;\\
\ang{g,g\inv\mid gg\inv = 1 = g\inv g} &\text{if } p=0.
\end{cases}
$$
Then by Theorem \ref{quiver-presentation},
$A*G$ is given by the following quivers with relations:
$$
A*G =
\begin{cases}
\xymatrix{1 \ar@(ul,dl)_{\al} \ar@(ur,dr)^{x}}, \quad x^p=0, \al^3=0,
\al x = x\al + x\al^2 +\al^2 &\text{if } p>0;\\
\xymatrix{1 \ar@(ul,dl)_{\al} \ar@(ul,ur)^{x} \ar@(dl,dr)_{x\inv}},
\quad xx\inv=1=x\inv x, \al^3=0,
\al x = x\al + x\al^2 &\text{if } p=0,
\end{cases}
$$
where we put $x:=(g,1) -1$ in the first case, and
$x:=(g,1)$ in the second case.
\end{exm}

\subsection{Brou\'{e}'s conjecture for $SL(2,4)$}
We can deal with the same example as in \cite[Example 6.2]{Asa97} by using the cyclic group $G:= \ang{g \mid g^2=1}$
instead of the infinite cyclic group.
\begin{exm}
Let $\La$ and $\Pi$ be the algebras given by
the following quivers with zero relations:
$$
\La :  \xymatrix{
2 \ar@<-1ex>[r]_{\al_{2}} & 1  \ar@<-1ex>[l]_{\al_{1}}\ar@<1ex>[r]^{\be_{1}} & 3 \ar@<1ex>[l]^{\be_{2}}
},\quad
\left\{
\begin{aligned}
&\be_{2}\be_{1}\al_{2}\al_{1} = \al_{2}\al_{1}\be_{2}\be_{1}\\
&\al_{1}\al_{2} = 0 = \be_{1}\be_{2}
\end{aligned}
\right.
$$

$$
\Pi : 
\vcenter{
\xymatrix@C8ex@R8ex{
& 1\ar@<-1ex>[dl]_{\al_{1}} \ar@<-1ex>[rd]_{\ga_{2}}\\
2 \ar@<-1ex>[ru]_{\al_{2}} \ar@<-2ex>[rr]_{\be_{1}} && 3 \ar@<-1ex>[lu]_{\ga_{1}} 
\ar[ll]_{\be_{2}}
}},\quad
\left\{
\begin{aligned}
\al_{2}\al_{1}=\ga_{1}\ga_{2}\\
\be_{2}\be_{1} = \al_{1}\al_{2}\\
\ga_{2}\ga_{1}= \be_{1}\be_{2}
\end{aligned} 
\right.
,\quad
\left\{
\begin{aligned}
\be_{1}\al_{1} = 0 = \al_{2}\be_{2}\\
\ga_{1}\be_{1} = 0 = \be_{2}\ga_{2}\\
\al_{1}\ga_{1}=0= \ga_{2}\al_{2}
\end{aligned} 
\right.
$$
As well known, when $\chr \k =2$, $\La$ is Morita equivalent to the principal block of the group algebra
$\k SL(2,4)$ and $\Pi$ is its Brauer correspondence; and
in this case Brou\'e's conjecture \cite{Br} claims that $\La$ and $\Pi$ are derived equivalent.  Here we show this fact without the assumption that
$\chr \k = 2$.
First we define $G$-gradings of $\La$ and $\Pi$ as follows:
$$
\deg(\al):=\begin{cases} g & \text{if $\al \in I$}\\
1 & \text{otherwise},\end{cases}
\quad\text{where }
I=\begin{cases}\{\al_{1}, \be_{1}\} & \text{for $\La$}\\
\{\be_{1}, \be_{2}\} & \text{for $\Pi$}.
\end{cases}
$$
A direct computation shows that
$\La \# G$ and $\Pi \# G$ has the form
\begin{equation}
\label{trivial-extension-form}
\La \# G \iso T(A),\quad \Pi \# G \iso T(B),
\end{equation}
where $A$ and $B$ are the algebras given by the following quivers with zero relations:
$$
A:=
\vcenter{
\xymatrix{
2 \ar@{--}@/^/[dd]\ar[rd]^{\al}&& 3 \ar[ld]_{\be}\ar@{--}@/_/[dd]\\
&1 \ar[rd]^{\de}\ar[ld]_{\ga}\\
5 \ar[rd]_{\ep}&& 6 \ar[ld]^{\ze}\\
&4
}},
\qquad
B:=
\vcenter{
\xymatrix{
&1 \ar[rd]\ar[ld]\ar@{--}@/^/[rddd]\ar[ld]\ar@{--}@/_/[lddd]\\
2 \ar[dd]&& 3 \ar[dd]\\
&4 \ar[rd]\ar[ld]\\
6 && 5
}}
$$
and for an algebra $C$, $T(C)$
denotes the trivial extension algebra
$C \ltimes DC$\ of $C$ by $DC:=\Hom_{\k}(C,\k)$.
Define an action of $g$ by $g(x):=x+3 \pmod 6$ both
on $T(A)$ and $T(B)$.
Then the isomorphisms in \eqref{trivial-extension-form} can be
taken to be strictly $G$-equivariant.
Hence by Proposition \ref{smash-orb} we have
\begin{equation}
\label{orbit-cat-form}
\La \simeq T(A)/G, \quad \Pi \simeq T(B)/G.
\end{equation}
Now define a full subcategory $E$ of $\Kb(\prj A)$ by the following
six objects: $T_{i}:=\udl{e_{i}A}$ ($i=2,3,5,6$), $T_{1}:=\udl{(e_{2}A\ds e_{3}A} \ya{(\al, \be)} e_{1}A)$,
and $T_{4}:=\udl{(e_{5}A\ds e_{6}A} \ya{(\ep,\ze)} e_{4}A)$, where the underline stands for the
place of degree zero.
Then $E$ is a tilting subcategory and an isomorphism $\ps\colon E \to B$ is defined by sending
$T_{i}$ to $i$ for all vertices $i=1,\dots, 6$ of the quiver of $B$.
This canonically induces a tilting subcategory $E'$ of $\Kb(\prj T(A))$ and
an isomorphism $\ps' \colon E' \to T(B)$ as in Rickard \cite{Ric2}.
As easily seen $\ps'$ can be taken to be $G$-equivariant, and therefore we see that
$T(A)/G$ and $T(B)/G$ are derived equivalent by Theorem \ref{thm:skewgp-cat}.
Hence by \eqref{orbit-cat-form} $\La$ and $\Pi$ are derived
equivalent, as desired.

In the above $T(A)$ and $T(B)$ were constructed from
$\La$ and $\Pi$ by taking smash products
as in \eqref{trivial-extension-form}.
If $\chr \k \ne 2$, the same thing can be done also by taking skew group algebras.
Indeed, define the $G$-actions on $\La$ and $\Pi$ as follows:
$g$ fixes all vertices, and 
$$
g(\al):= \begin{cases} -\al & \text{if $\al\in I$}\\
\al & \text{otherwise},\end{cases}
\quad\text{where }
I=\begin{cases}\{\al_{1}, \be_{1}\} & \text{for $\La$}\\
\{\be_{1}, \be_{2}\} & \text{for $\Pi$}.
\end{cases}
$$
Then $\La * G \iso T(A)$ and $\Pi * G \iso T(B)$.
\end{exm}

\subsection{Derived equivalence}
\begin{exm}
Here assume that $\chr \k =0$.
Let $G =\ang{g,g\inv\mid gg\inv = 1 = g\inv g}$ be the infinite cyclic group.
Define algebras $A$ and $B$ as follows:
$$
A : \xymatrix{
1 \ar@<1ex>[r]^{\al_{1}} & 2  \ar@<1ex>[l]^{\al_{2}}\ar@<1ex>[r]^{\be_{1}} & 3 \ar@<1ex>[l]^{\be_{2}}
},\quad
\left\{
\begin{aligned}
&\al_{i}\be_{j} = 0 = \be_{i}\al_{j}\text{ for all $i, j$}\\
&\al_{1}\al_{2} =  (\be_{2}\be_{1})^{2},
\end{aligned}
\right.
$$

$$
B :
\vcenter{
\xymatrix{
& 1\ar[dl]_{\al_{1}}\\
2 \ar[rr]_{\al_{2}} && 3 \ar[lu]_{\al_{3}} 
}},\quad
\al^{7}=0 \ (\text{paths of length 7} = 0).
$$
Then $A$ and $B$ are derived equivalent by a tilting subcategory $E$ of $\Kb(\prj A)$
defined as follows.
$$
\xymatrix{
& (\udl{e_{2}A}\ya{\al_{2}} e_{1}A) \ar[dl]_{(\id,0)}\\
\udl{e_{2}A} \ar[rr]_{(\be_{1},0)} && \udl{e_{3}A} \ar[lu]_{(\be_{2},0)} 
}
$$
We have an obvious isomorphism $\ps\colon E \to B$.
Now define $G$-actions on $A$ and $B$ as follows.

On $A$: $g$ fixes all vertices and all $\al_{i}$, and $g(\be_{i}):= \be_{i}+\be_{i}\be_{i+1}\be_{i}$
for all $i$.

On $B$: $g$ fixes all vertices and $\al_{1}$, and $g(\al_{i}):=\al_{i}+\al_{i}\al_{i+2}\al_{i+1}\al_{i}$
$\pmod 3$ for $i \ne 1$.
Then as easily seen $\ps$ is $G$-equivariant, and hence $A*G$ and $B*G$ are derived equivalent.
Here $A*G$ and $B*G$ are presented as follows.
$$
A*G :
\xymatrix{
1 \ar@(ul,ur)^{x}\ar@(dr,dl)^{x\inv} \ar@<1ex>[r]^{\al_{1}}
& 2 \ar@(ul,ur)^{y}\ar@(dr,dl)^{y\inv} \ar@<1ex>[l]^{\al_{2}}\ar@<1ex>[r]^{\be_{1}} & 3 \ar@(ul,ur)^{z}\ar@(dr,dl)^{z\inv}\ar@<1ex>[l]^{\be_{2}}
},\quad
\left\{
\begin{aligned}
&\al_{i}\be_{j} = 0 = \be_{i}\al_{j}\text{ for all $i, j$},\ \al_{1}\al_{2} =  (\be_{2}\be_{1})^{2}\\
&xx\inv = 1 = x\inv x, yy\inv=1=y\inv y, zz\inv=1=z\inv z\\
&\al_{1}x=y\al_{1}, y\al_{2}=\al_{2}x, \\
&\be_{1}y=z\be_{1}+z\be_{1}\be_{2}\be_{1}\\
&\be_{2}z=y\be_{1}+y\be_{2}\be_{1}\be_{2},
\end{aligned}
\right.
$$

$$
B*G :
\vcenter{
\xymatrix{
& 1  \ar@(dl,ul)^{x}\ar@(dr,ur)_{x\inv} \ar[dl]_{\al_{1}}\\
2  \ar@(ur,ul)_{y}\ar@(dr,dl)^{y\inv} \ar[rr]_{\al_{2}} &&
3 \ar@(ul,ur)^{z}\ar@(dr,dl)^{z\inv} \ar[lu]_{\al_{3}} 
}},\quad
\left\{
\begin{aligned}
&\al^{7}=0, xx\inv = 1 = x\inv x, yy\inv=1=y\inv y, zz\inv=1=z\inv z\\
&\al_{1}x=y\al_{1}\\
&\al_{2}y=z\al_{2}+z\al_{2}\al_{1}\al_{3}\al_{2}\\
&\al_{3}z=x\al_{3}+x\al_{3}\al_{2}\al_{1}\al_{3}.
\end{aligned}
\right.
$$
\end{exm}

\subsection{Preprojective algebra, monoid case}
\begin{exm}
Let $Q$ be the following quiver of type $A_{4}$:
$$
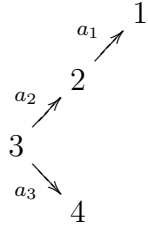
\xymatrix@C2ex@R2ex{
&& 1\\
&2 \ar[ru]^{a_{1}}\\
3\ar[ru]^{a_{2}} \ar[rd]_{a_{3}}\\
& 4
}
$$
and let $A:=\k Q$.
Then the Auslander-Reiten quiver $\Ga_{A}$ is as follows.
$$
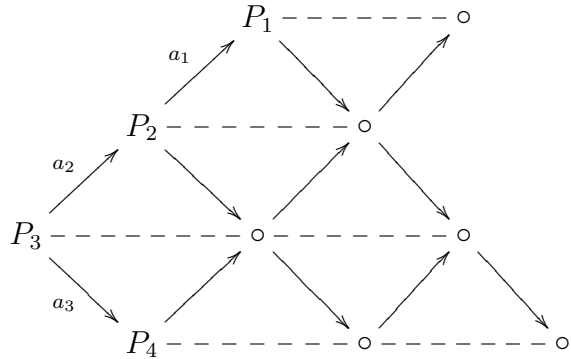
\xymatrix{
&& P_{1} \ar[rd]\ar@{--}[rr] &&\circ\\
&P_{2} \ar[rd]\ar[ru]^{a_{1}} \ar@{--}[rr] &&\circ\ar[rd]\ar[ru]\\
P_{3}\ar[ru]^{a_{2}} \ar[rd]_{a_{3}}\ar@{--}[rr] &&\circ \ar[ru]\ar[rd]\ar@{--}[rr] &&\circ\ar[rd]\\
& P_{4} \ar[ru]\ar@{--}[rr] &&\circ \ar[ru]\ar@{--}[rr] &&\circ
}
$$
Then $\mod A$ is equivalent to the additive hull $\add \k(\Ga_{A})$ of the mesh
category  $\k(\Ga_{A})$ of $\Ga_{A}$.
Let $G:=\ang{\ta\inv\mid \ta^{-3}=0}$, which is a monoid with zero.
By definition the preprojective algebra $\calP(Q)$ of $Q$ is given by
$$
\begin{aligned}
\calP(Q)&:=\Ds_{n\ge 0}(\mod A)(A, \ta^{-n}A)\\
 &\iso\Ds_{n\ge 0}(\add \k(\Ga_{A}))(A, \ta^{-n}A)\\
 &\iso (\add \k(\Ga_{A}) \sndorbit G)(A, A),
\end{aligned}
$$
where $\k(\Ga_{A})\sndorbit G \iso (\k(\Ga_{A})\op \fstorbit G)\op
\iso \k(\Ga'_{A})/I$.
Here $\Ga'_{A}$ is given by
$$
\xymatrix{
&& P_{1} \ar[rd]\ar@{<-}[rr] &&\circ\\
&P_{2} \ar[rd]\ar[ru]^{a_{1}} \ar@{<-}[rr] &&\circ\ar[rd]\ar[ru]\\
P_{3}\ar[ru]^{a_{2}} \ar[rd]_{a_{3}}\ar@{<-}[rr] &&\circ \ar[ru]\ar[rd]\ar@{<-}[rr] &&\circ\ar[rd]\\
& P_{4} \ar[ru]\ar@{<-}[rr] &&\circ \ar[ru]\ar@{<-}[rr] &&\circ
}
$$
and $I$ is generated by all mesh relations and commutativity relations
$x_j a = (\ta a) x_i$ for all old arrows $a: i \to j$ such that $\ta a$ exists,
where $x_k$ is the new arrows starting at a vertex $k$ and $\ta a$ is the Auslander-Reiten translation
of an old arrow $a$.
By computing the endomorphism algebra of $A$ ($=P_1\ds\cdots\ds P_4$)
inside this category we get
$$
\calP(Q):
\vcenter{
\xymatrix@C2ex@R2ex{
&& 1\ar@/^/[dl]^{a'_{1}}\\
&2 \ar[ru]^{a_{1}}\ar@/^/[dl]^{a'_{2}}\\
3\ar[ru]^{a_{2}} \ar[rd]^{a_{3}}\\
& 4 \ar@/^/[ul]^{a'_{3}}
}}, \quad
\left\{
\begin{aligned}
&a_{1}a'_{1}=0,\\
&a'_{1}a_{1}+a'_{2}a_{2}=0,\\
&a'_{3}a_{3}+a'_{2}a_{2}=0,\\
&a_{3}a'_{3}=0.
\end{aligned}
\right.
$$
\end{exm}

\subsection{Nakayama permutations}\label{Nak-perm}
K.~Oshiro asked the following problem to the author in March, 2008.
We give an answer to it using the classical covering technique and
Theorem \ref{quiver-presentation} along the line of Remark \ref{computation-orbitcat}(2).

\begin{prb-nn}
For each permutation $\si\in S_n$ of the set $\{1,\dots, n\}$,
construct a self-injective algebra $A$ whose Nakayama permutation is $\si$,
and if possible give such an example by an algebra with radical cube zero.
\end{prb-nn}

First decompose the $\si$ into a product of cyclic permutations:
$$\si=(x_{11}\ x_{12}\ \cdots x_{1,t(1)})\cdots (x_{m1}\ x_{12}\ \cdots x_{m,t(m)})$$
such that $\{1,\dots, n\}= \{x_{11}\ x_{12}\ \cdots x_{1,t(1)}\}\cup \{x_{m1}\ x_{12}\ \cup x_{m,t(m)}\}$ is a disjoin union
(we allow $t(i)=1$ here). Then $t(1)+\cdots +t(m)=n$.
Further we set
$x_{i,t(i)+1}:=x_{i1}$ (for all $i$) and consider $j$ in $x_{ij}$
modulo $t(i)$.

\begin{exm}
For instance,
for $\si:=\pmat{1&2&3&4\\1&2&4&3}=(1)(2)(3\ 4) \in S_4$,
we have $\si=(x_{11})(x_{21})(x_{31}\ x_{32})$ with
$t(1)=1, t(2)=1, t(3)=2$;
$x_{11}=1, x_{21}=2, x_{31}=3, x_{32}=4$.
\end{exm}

Next define a quiver $Q:=(Q_0, Q_1)$ as follows.
\begin{align*}
Q_0:=&\{1,\ldots, n\}= \bigcup_{i=1}^m\{x_{i1},\ldots, x_{i,t(i)}\}\\
Q_1:=&\{\al_{ijl} \mid 1\le i \le m-1, i\not\in 2\bbZ, 1 \le j \le t(i), 1\le l \le t(i+1)\}\\
\cup &\{\be_{ijl} \mid 1\le i \le m-1, i\in 2\bbZ, 1 \le j \le t(i), 1\le l \le t(i+1)\}
\end{align*}
with orientations
\begin{equation}\label{orientation}
\xymatrix{
x_{ij} \ar[r]^{\al_{ijl}} & x_{i+1,l}
},
\ 
\xymatrix{
x_{ij} & \ar[l]_{\be_{ijl}} x_{i+1,l}
}.
\end{equation}

For instance in the example above, we have
$$
\xymatrix{
x_{11} \ar[r]^{\al_{111}} &x_{21} &\ar[l]_{\be_{211}} x_{31}\\
&&\ar[lu]^{\be_{212}} x_{32}
}.
$$

Then $\si$ can be regarded as a permutation of $Q_0$ and it is uniquely
extended to an automorphism of the quiver $Q$.
By identifying $\si$ with the linearization of this, we can regard
$\si$ as an automorphism of the path-algebra $\k Q$.
Further $\si$ is canonically extended to an automorphism $\hat{\si}$
of the repetition $\widehat{\k Q}$.

\begin{thm}
Let $A$ be the twisted 1-fold extension of $\k Q$ by $\si$
$($\cite{Asa02}$)$, namely
$A:=T_{\si}^1(\k Q):=\widehat{\k Q}/\ang{\nu\hat{\si}}$, where
$\nu$ is the Nakayama automorphism of $\widehat{\k Q}$.
Then $A$ is a self-injective algebra with radical cube zero and
$\si$ is its Nakayama permutation.
\end{thm}

\begin{proof}
The orbit category construction above is the classical one.
As explained in Remark \ref{computation-orbitcat}(2) we can use
Theorem \ref{quiver-presentation} to compute $A =\widehat{\k Q}\oorbit \ang{\nu\hat{\si}}$,
and we see that
$A$ has the following presentation by a quiver with relations:
The quiver $Q_A:=(Q'_0, Q'_1, t', h')$ of $A$ is defined as follows.
$Q_0=Q'_0$, $Q'_1=\{\al_{ijl}, \be_{ijl} \mid
1\le i \le m-1, 1 \le j \le t(i), 1\le l \le t(i+1)\}$
and the orientations of $\al_{ijl}, \be_{ijl}$ are defined
by \eqref{orientation}; and relations are given by zero relations and
commutativity relations below.

{\bf zero relations:}

$\al_{ijl}\al_{rst}=0; \be_{ijl}\be_{rst}=0,$
for $\forall i,j,l,r,s,t$;

$\be_{ijl}\al_{rst}=0$ unless $(r,s,t)=(i,j,l+1)$;
$\al_{ijl}\be_{rst}=0$ unless $(r,s,t)=(i,j+1,l)$;

{\bf commutativity relations:}

$\al_{i-1,p,j+1}\be_{i-1,p,j} = \be_{i,j+1,l}\al_{ijl}$ ($2 \le i \le m-1$);

$\be_{i,j+1,l}\la_{ijl}=\be_{i,j+1,p}\la_{ijp}$ ($1 \le i \le m-1$);

$\al_{i-1,l,j+1}\be_{i-1,l,j} = \al_{i-1,p,j+1}\be_{i-1,p,j}$ ($2 \le i \le m$).

This shows that the indecomposable projective modules
$P(x_{ij}):=Ae_{x_{ij}}$ have the following structures
for all $x_{ij}\in Q'_0$:
$$
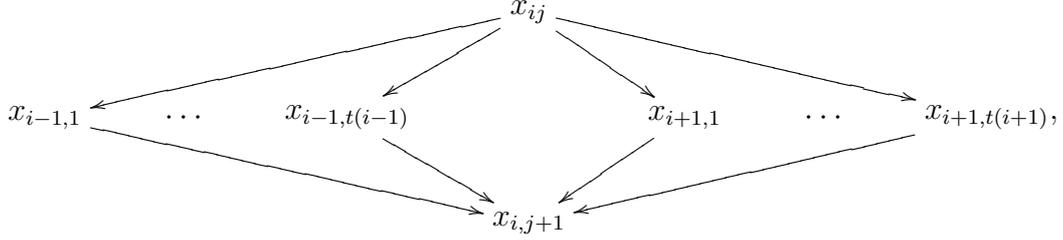
\xymatrix{
&&& x_{ij}\ar[llld]\ar[ld]\ar[rd]\ar[rrrd]\\
x_{i-1,1}\ar[rrrd]&\cdots&x_{i-1,t(i-1)}\ar[rd] && x_{i+1,1}\ar[ld]&\cdots&x_{i+1,t(i+1)}, \ar[llld]\\
&&&x_{i,j+1}
}
$$
where for $i=1$ delete the left side part
$x_{i-1,1},\ldots,x_{i-1,t(i-1)}$ and
for $i=m$ delete the right side part
$x_{i+1,1},\ldots,x_{i+1,t(i+1)}$.
Therefore $A$ has the radical cube zero and
$\soc P(x_{ij})\iso \top P(x_{i,j+1})$, and hence
$A$ is a self-injective algebra with Nakayama permutation $\si$.
\end{proof}

For instance in the example above $Q_A$ has the form
$$
\xymatrix{
x_{11} \ar@<1ex>[r]^{\al_{111}} &x_{21} \ar[l]^{\be_{111}}\ar@<1ex>[r]^{\al_{211}}\ar[rd]^(.7){\al_{212}}
&\ar[l]^(.3){\be_{211}} x_{31}\\
&&\ar@<1ex>[lu]^{\be_{212}} x_{32}
}
\quad
\text{or more simply}
\quad
\xymatrix{
1 \ar@<1ex>[r]^{\al_{1}} &2 \ar[l]^{\be_{1}}\ar@<1ex>[r]^{\al_{2}}\ar[rd]^(.7){\al_{3}}
&\ar[l]^(.3){\be_{2}} 3\\
&&\ar@<1ex>[lu]^{\be_{3}} 4
}
$$
and the structure of projective indecomposables are as follows:
$$
\pmat{1\\2\\1} \pmat{&2&\\1&3&4\\&2&}
\pmat{3\\2\\4} \pmat{4\\2\\3}.
$$

\begin{exm}
For $\si=(1\ 2)(3\ 4)$, $Q_A$ and its projective indecomposables
are as follows:
$$
\vcenter{
\xymatrix{
1\ar@<.5ex>[r]\ar@<.5ex>[rd] & 3\ar@<.5ex>[l]\ar@<.5ex>[ld]\\
2\ar@<.5ex>[r]\ar@<.5ex>[ru]  & 4\ar@<.5ex>[l]\ar@<.5ex>[lu]
}};
\quad
\pmat{&1&\\3&&4\\&2&} \pmat{&2&\\3&&4\\&1&} \pmat{&3&\\1&&2\\&4&} \pmat{&4&\\1&&2\\&3&}.
$$
\end{exm}

\section*{Acknowledgments}
I would like to thank Bernhard Keller for answering my questions, in particular for giving me
an account of the construction of the orbit category $\C \sndorbit G$.
I would also like to thank Alex Dugas and Kiyoichi Oshiro for their questions, the former led me to smash products defined in \cite{C-M},
and the latter enabled me to include section \ref{Nak-perm}.
Finally, I would like to thank the referee for pointing out some careless mistakes
and for some suggestions to improve the paper.

\end{document}